\documentclass[12pt]{article}

\usepackage{amsrefs}

\usepackage{amsfonts, amsmath, amssymb, amsthm, color, hyperref, enumitem, fancyhdr, relsize, tikz, graphicx, mathrsfs}
\usepackage[utf8]{inputenc}
\usepackage{lipsum}

\setlist[description]{leftmargin=1cm,labelindent=0.5cm}

\usepackage{caption}
\usepackage{subcaption}

\newtheorem{prop}{Proposition}[section]

\newtheorem{thm}[prop]{Theorem}
\newtheorem{lemma}[prop]{Lemma}

\theoremstyle{definition}
\newtheorem{remark}[prop]{Remark}
\newtheorem{example}[prop]{Example}

\newtheorem{defn}[prop]{Definition}

\newcommand{\R}{\mathbb{R}}
\newcommand{\C}{\mathbb{C}}

\newcommand{\bS}{\mathbb{S}}

\renewcommand{\Im}{\operatorname{Im}}
\renewcommand{\Re}{\operatorname{Re}}

\renewcommand{\H}{{\mathcal H}^{N-1}}
\renewcommand{\L}{{\mathscr L}}

\newcommand{\diag}{\operatorname{diag}}

\newcommand{\dist}{\operatorname{dist}}

\newcommand{\Tr}{\operatorname{Tr}}
\newcommand{\bp}{\begin{pmatrix}}
\newcommand{\ep}{\end{pmatrix}}

\newcommand{\A}{{\mathcal A}}

\renewcommand{\H}{{\mathcal H}}

\newcommand{\dark}{}

\title{\scshape 
Applications of the Projection Characterization and Symmetry Bootstrap for Operators that are Nearby Commuting Operators
\sffamily }
\author{\scshape \sffamily David Herrera}

\newcommand{\Addresses}{{
  \bigskip
  \footnotesize

  D.~Herrera, \textsc{Department of Mathematics, Rowan College at Burlington County,
     Mount Laurel, New Jersey 08054 }\par\nopagebreak
  \textit{E-mail address}, D.~Herrera: \texttt{dherrera@rcbc.edu}

}}

\begin{document}
\dark
\large

\maketitle

\abstract{This paper applies the Projection Characterization and the Symmetry Bootstrap for two operators $A, B$ (with $A$ unitary or self-adjoint) nearby commuting operators $A', B'$ (with $A'$ unitary or self-adjoint) from \cite{herrera2024projection} to problems related to structured almost commuting unitaries, self-adjoint operators that almost anti-commute, and unitaries that almost commute up to a rational phase.

We prove several new stability theorems for these types of relations for matrices. We also provide estimates for how close the constructed operators are for several known stability theorems that did not previously have asymptotic estimates.
}

\section{Introduction}

Lin's Theorem in the case of matrices states:
\begin{thm}(\cite{Lin1996almost})
There is a function $\epsilon(\delta
)$ with $\lim_{\delta \to 0^+}\epsilon(\delta)=0$ such that for all $n$ and all self-adjoint contractions $X, Y \in M_n(\C)$, there exist commuting self-adjoint $X', Y' \in M_n(\C)$ such that
\[\|X'-X\|, \|Y'-Y\| \leq \epsilon(\|[X,Y]\|),\]
where $[X,Y]=XY-YX$ is the commutator of $X$ and $Y$ and $\|-\|$ is the operator norm.
\end{thm}
This result is non-trivial because the estimate for how close the nearby commuting matrices $X', Y'$ are to the original matrices $X, Y$ does not depend on $n$. 

See \cites{herrera2022constructing, herrera2024constructing} and the references therein for a prehistory of Lin's Theorem. Some notable results are that it is not true that a dimensional-independent estimate for nearby commuting matrices can be found for certain examples, such as three or more almost commuting matrices (\cite{
voiculescu1981remarks, davidson1985almost, choi1988almost}) or for more general types of almost commuting matrices, such as a self-adjoint and normal matrix (\cite{davidson1985almost}) or two almost commuting unitaries (\cite{voiculescu1983asymptotically}).

There has been much work on extending Lin's theorem since it was proved, for instance 
\cites{
dor2025almost,
eilers1999morphisms,
enders2019almost,
friis1996almost,
gong1998almost,
hastings2009making,
hastings2011making,
hastings2010almost,
hastings2011topological,
herrera2022constructing,
herrera2020hastings,
herrera2024projection,
kachkovskiy2016distance,
li2022vector,
lin2024almost,
loring1988k,
loring2015k,
loring1997lifting,
loring2016almost,
loring2014almost,
loring2013almost,
ogata2013approximating
}. 
Some of these generalizations involve extending the types of $C^\ast$-algebras in which which there exist nearby commuting elements, obtaining explicit estimates for $\epsilon(\delta)$, or extending Lin's theorem to different types of almost commuting operators when necessary obstructions are not present.

A very notable example is that Kachkovskiy and Safarov (\cite{kachkovskiy2016distance}) showed that there is a constant $C_{KS}$ so that we can choose $X', Y'$ in Lin's theorem to satisfy
\begin{equation}\label{KS}
\|X'-X\|, \|Y'-Y\| \leq C_{KS}\|[X,Y]\|^{1/2}.
\end{equation}

Another notable example that much of our paper is devoted to extending is the result that almost commuting unitaries for which the necessary obstruction vanishes are nearby commuting unitaries (\cite{eilers1999morphisms}, \cite{gong1998almost}).  Recently in 2025, Dor-On, Hall, and Kachkovskiy in \cite{dor2025almost} provided the estimate
\begin{equation}\label{Dor-On Unitary Ineq}
\|U'-U\|, \|V'-V\| \leq Const. \|[U,V]\|^{2/15}.
\end{equation}

The purpose of this paper is to explore several applications of the Projection Characterization and the Symmetry Bootstrap of operators that are nearby commuting operators from our 2024 paper
\cite{herrera2024projection}. We prove the following extension of the above almost commuting unitary result:
\begin{thm}\label{unitary thm}
If $U, V \in M_n(\C)$ are unitaries that have a Bott index equal to zero (see: Definition \ref{bott index}) then there are commuting unitaries $U', V' \in M_n(\C)$ such that
\[\|U'-U\|, \|V'-V\| \leq Const. \|[U,V]\|^{1/8}.\]

Moreover, if $U, V$ have additional structure as in one of the following cases then $U', V'$ can be chosen with that same structure. Let $W\in M_n(\C)$ be unitary, $m \geq 2$ be an integer, $\zeta, \eta \in \C$ in the unit circle, and the value of $Const.$ depend on $m$:
\begin{description} 
\item[\underline{$(\zeta,\eta)$}:] Suppose further that $\zeta$ has order $m$ (as an element of the multiplicative group $\C^\times$) and $W^m$ is a multiple of $I$.\\
Then if $WUW^\ast=\zeta U$ and $WVW^\ast=\eta V$ (resp. $WVW^\ast=\eta V^\ast$) then $U', V'$ can be chosen to also satisfy $WU'W^\ast=\zeta U'$ and $WV'W^\ast=\eta V'$ (resp. $WV'W^\ast=\eta (V')^\ast$). 
\item[\underline{$(\zeta,\eta)$T}:] Suppose further that $\zeta$ has order $m$ even and $(\overline{W}W)^{m/2}$ is a multiple of $I$.\\
Then if $WUW^\ast=\zeta U^T$ and  $WVW^\ast=\eta V^T$ (resp. $WVW^\ast=\eta \overline{V}$) then $U', V'$ can be chosen to also satisfy $WU'W^\ast=(U')^T$ and  $WV'W^\ast=\eta (V')^T$ (resp. $WV'W^\ast=\eta \overline{V'}$). 
\item[\underline{$(\zeta,\eta)$C}:] Suppose further that $m=2$ and $\overline{W}W$ is a multiple of $I$.\\
Then if $WUW^\ast=\zeta \overline{U}$ and $WVW^\ast=\eta \overline{V}$ (resp. $WVW^\ast=\eta V^T$) then $U', V'$ can be chosen to also satisfy $WU'W^\ast=\zeta \overline{U'}$ and  $WV'W^\ast=\eta \overline{V'}$ (resp. $WV'W^\ast=\eta (V')^T$). 
\item[\underline{$(\zeta,\eta)$A}:] Suppose further that $m=2$ and $W^2$ is a multiple of $I$. 

Then if $WUW^\ast=\zeta U^\ast$ and  $WVW^\ast=\eta V^\ast$ (resp. $WVW^\ast=\eta V$) then $U', V'$ can be chosen to also satisfy $WU'W^\ast=\zeta (U')^\ast$ and  $WV'W^\ast=\eta (V')^\ast$ (resp. $WV'W^\ast=\eta V'$). 
\end{description}
\end{thm}

\begin{remark}
Recall that as shown in \cite{hastings2010almost}, two unitary matrices that are real are guaranteed to have nearby commuting unitaries because the Bott index vanishes. It was not until \cite{loring2014almost} that it was shown that the nearby unitaries can be chosen to be real. 

This is now a direct consequence of our bootstrap theorem because a real unitary is conjugation-symmetric. That is, it is a consequence of the above result with \underline{(1,1)\textbf{C}} above with $\zeta=\eta=1$ and $W = I$. Moreover, we obtain an asymptotic estimate whereas before there was none.

We also obtain a similar result for two almost commuting symplectic unitary matrices since as explained in \cite[Example 2.4(8)]{herrera2024projection} they satisfy the \underline{(1,1)\textbf{C}} symmetry with $\zeta=\eta=1$ and a certain real unitary $W$ that satisfies $W^2=-I$. 
\end{remark}

We now discuss some new results of a similar type that we obtain for almost anti-commuting self-adjoint matrices and unitaries that almost commute up to a rational phase.

Hua and Lin in \cite{hua2015rotation} studied unitaries in an abstract setting that nicely generalizes unitary matrices that almost commute up to a phase. They show that when the appropriate obstruction vanishes for unitaries $U$ and $V$ that satisfy $UV \approx e^{2\pi i\theta} VU$, there are nearby unitaries $U', V'$ that exactly satisfy $U'V'= e^{2\pi i\theta} V'U'$.
There are some extensions of this to more than two unitaries in \cite{hua2021stability} subject to certain restrictions.

In Theorem \ref{almost commuting unitaries up to a phase}, we extend the result in \cite{hua2015rotation} by providing a construction for the unitaries and an estimate for how close they may be when the obstruction vanishes:
\begin{equation}\label{phaseKS}
\|U'-U\|, \|V'-V\| \leq C_\theta\,\|UV-e^{2\pi i \theta}VU\|^{1/16},
\end{equation}
where $\theta$ is a rational number and $C_\theta$ depends on the denominator of $\theta$. We prove this result by simply applying our result constructing commuting unitaries and a symmetry bootstrap to an associated pair of structured block unitary matrices that approximately commute. 

\vspace{0.1in}

Also, because of the connection between almost commuting unitaries and almost commuting unitaries up to a rational phase that we prove in Theorem \ref{Up to Phase Bootstrap}, we could just prove stability results of rational phase almost commutativity immediately from other known results.

For instance, we could obtain a similar version of (\ref{phaseKS}) above by simply applying our Theorem \ref{Up to Phase Bootstrap} to (\ref{Dor-On Unitary Ineq}).
Another example is that we apply our Theorem \ref{Up to Phase Bootstrap} to (\ref{KS}) to prove the following result which generalizes \cite{pedersen1998stability} by providing estimates:
\begin{thm}
There is a universal constant $C> 0$ such that if $X, Y\in M_n(\C)$ are self-adjoint contractions, then there are anti-commuting $X', Y'$ self-adjoint contractions with
\[\|X'-X\|, \|Y'-Y\| \leq C\|XY+YX\|^{1/8}.\]
\end{thm}

Another generalization we obtain as a corollary of Loring and S{\o}rensen's proof of Lin's theorem for real or self-dual matrices in \cite{loring2016almost} is the anti-commuting version of that result:
\begin{thm}
There is a function $\epsilon(\delta
)$ with $\lim_{\delta \to 0^+}\epsilon(\delta)=0$ such that for all $n$ and all real (resp. self-dual) self-adjoint contractions $X, Y \in M_n(\C)$, there exist anti-commuting real (resp. self-dual) self-adjoint $A', B' \in M_n(\C)$ such that
\[\|X'-X\|, \|Y'-Y\| \leq \epsilon(\|XY+YX\|).\]
\end{thm}

\vspace{0.1in}

We also apply these methods to infinite expansions of unitaries that almost commute up to a phase. For some context, Lin proved the following
\begin{thm} (\cite{lin1995almost})
There is a function $\epsilon(\delta
)$ with $\lim_{\delta \to 0^+}\epsilon(\delta)=0$ such that for all $n$ and $U, V \in M_n(\C)$, there exist commuting unitary operators $\mathcal U', \mathcal V'$ such that
\[\|\mathcal U'-U\otimes I_{\infty}\|, \|\mathcal V'-V\otimes I_{\infty}\| \leq \epsilon(\|[U,V]\|),\]
where $I_\infty$ is the identity on a countably infinite dimensional Hilbert space.
\end{thm}

As a corollary of this and our Theorem \ref{Up to Phase Bootstrap}, we immediately obtain
\begin{thm}\label{expansion}
Let $\omega$ be an $m$th root of unity.
There is a function $\epsilon(\delta
)$ that depends on $m$ with $\lim_{\delta \to 0^+}\epsilon(\delta)=0$ such that for all $n$ and $U, V \in M_n(\C)$, there exist unitary operators $\mathcal U', \mathcal V'$ such that
\[\mathcal U'\mathcal V' = \omega \mathcal V'\mathcal U',\]
\[\|\mathcal U'-U\otimes I_{\infty}\|, \|\mathcal V'-V\otimes I_{\infty}\| \leq \epsilon(\|UV- \omega VU\|),\]
where $I_\infty$ is the identity on a countably infinite dimensional Hilbert space $\mathcal H$.
\end{thm}

\begin{remark}\label{Infinite Dilation}
This provides the complement of the result in \cite{gerhold2024dilation} showing the stability of the relation $UV = \omega VU$ given a dilation by $I_\infty$, when $\omega$ is not a root of unity. 

It also appears that this result is also a corollary of \cite[Theorem 6.4]{hua2021stability}.
\end{remark}

\begin{remark}
Much of this current paper was written prior to when the first preprint of \cite{dor2025almost} was posted to the ArXiv in 2025. Our current paper was originally part of \cite{herrera2024projection}, which was broken off in 2024 as an independent paper with the current paper focusing on the applications of \cite{herrera2024projection}.

The methods used in this paper to construct almost commuting unitaries (without considering symmetries) are related to those of \cite{dor2025almost}. However, we do not reframe the problem in terms of homotopies of unitaries.

In particular, our approach involves using our Projection Characterization from \cite{herrera2024projection} to reduce the almost commuting unitaries problem to the almost representation of the sphere whose solution has been known for a while. Similar considerations are made in \cite{dor2025almost}.

As to other similarities, our Projection Characterization blackbox from the prior work \cite{herrera2024projection} uses projection inequalities in its Section 5 which are similar to those later used in Section 2.1 of \cite{dor2025almost}.
Likewise, the decomposition at the end of the proof of Lemma 2.12 of \cite{dor2025almost} bears some resemblance to our decomposition of the Bott block matrix below in Lemma \ref{index localization}.
Also, their decomposition of intertwined projections in Proposition 2.13 of \cite{dor2025almost} is similar to the use of Lemma 6.1 in Lemma 6.4 of our prior work \cite{herrera2024projection}. 
\end{remark}

\section{Almost Commuting Unitaries}

In this section we review the relevant known results for almost commuting unitary matrices. We refer to \cite{hastings2010almost} and \cite{loring2014quantitative}.

\begin{defn}(\cite{hastings2010almost})
We say that the self-adjoint $H_1, H_2, H_3\in M_n(\C)$ form a $\delta$-almost representation of the sphere if
\[\|H_1^2+H_2^2+H_3^2-I\|\leq \delta, 
\;\; \|[H_i, H_j]\|\leq \delta.\]
We say that $H_1, H_2, H_3$ form an exact representation of the sphere if the above hold for $\delta=0$.
\end{defn}
\begin{defn}
We define the Bott matrix 
\[B(H_1, H_2, H_3) = \bp H_1 & H_2+iH_3 \\H_2-iH_3 & -H_1 \ep\] and the Bott invariant/index of $H_1, H_2, H_3$ to be the signature of $B(H_1, H_2, H_3)$: the number of positive eigenvalues minus the number of negative eigenvalues. 
\end{defn}
One can show that if $H_1, H_2, H_3$ almost represent the sphere then $B(H_1, H_2, H_3)$ has square close to $I \oplus I$.
So, for $\delta$ small enough, the Bott index is well-defined and  stable under small perturbations due to eigenvalue perturbation theory. Because we are not particularly interested in obtaining an estimate with an explicit constant in our later results, we will settle for reviewing the known results with the same level of analysis.

Now, it is straightforward to show that the Bott index vanishes for an exact representation of the sphere. Loring in \cite{loring1998matrices} showed that an almost representation of the sphere is nearby an exact representation of the sphere if and only if the Bott index of the almost representation vanishes. Hastings and Loring in \cite{hastings2010almost} used the estimates proposed in \cite{hastings2009making} to obtain an asymptotic estimate for how close such an almost representation of the sphere is to an exact representation of the sphere.

Although Kachkovskiy and Safarov's result did not exist at the time, their work (instead of \cite{hastings2009making}) in combination with the construction in \cite{hastings2010almost} shows:
\begin{thm}\label{NearbyExactRep}
Suppose that $H_1, H_2, H_3$ are a $\delta$-almost representation of the sphere with vanishing Bott index. Then there are $H_1', H_2', H_3'$ that exactly represent   the sphere with 
\[\|H_i'-H_i\| \leq Const.\delta^{1/4}.\]
\end{thm}

We now review the construction of the Bott index for two almost commuting unitary matrices.
Let $f, g, h$ be real-valued operator Lipschitz functions on the unit circle with $f^2+g^2+h^2=1$. Then consider the Bott matrix for two almost commuting unitaries $U, V$ as defined in \cite{loring2014quantitative}:
\begin{align}\label{B^2}
B(U,V) = \bp f(V) & g(V)+\frac12(Uh(V)+h(V)U)\\
g(V)+\frac12(U^\ast h(V)+h(V)U^\ast) & -f(V)
\ep.
\end{align}
Note the notation for the anti-commutator: $\{A,B\} = AB+BA$.

\begin{remark}
There are different versions of $B(U,V)$, but this convention is best for our purposes because if $U, V$ satisfy a linear symmetry then one can show that the entries of $B(U,V)$ also satisfy this linear symmetry because the product $\{A,B\}$ behaves well under anti-multiplicative symmetries. 

Note that we choose the convention that the 2,2 entry of the block matrix be $-f(V)$ and not $1-f(V)$ as it is in some papers. This makes $B(U,V)$ almost an idempotent and we require $f^2+g^2+h^2=1$. The other convention is slightly different on the requirement for $f$ and it makes $B$ almost a projection. 
\end{remark}

By Lemma 5.3 of \cite{loring2014quantitative}, $B(U,V)$ is self-adjoint and \[\|B(U,V)^2-I\oplus I\| \leq 2\|[h(V), U]\|+\|[f(V), U]\|.\]
\begin{defn}\label{bott index}
The Bott invariant/index of $U, V$ is the signature of $B(U,V)$. This is well-defined when $\|[U,V]\|$ small enough and is stable under small perturbations of $U, V$. 
\end{defn}
Two almost commuting unitaries are also referred to as an almost representation of the torus.

Now, consider 
\[H_1 = f(V), \;
H_2 = \Re\left(g(V)+\frac12\{U,h(V)\}\right),\; H_3 = \Im\left(g(V)+\frac12\{U,h(V)\}\right).\]
Then $H_1$, $H_2$, $H_3$ are almost commuting self-adjoint matrices with $H_1^2+H_2^2+H_3^2 \approx I$. The $H_1, H_2, H_3$ then form an almost representation of the sphere. So, we have taken an almost representation of the torus and produced an almost representation of the sphere which has the same Bott index, appropriately defined.

There are other equivalent forms of the Bott invariant for two almost commuting unitaries such as a $K$-theory obstruction, the Exel-Loring winding number invariant from \cites{exel1989almost, exel1991invariants}, and the Exel trace formula (\cite{exel1993soft}):
\[\frac{1}{2\pi}\Tr[\log(U^{-1}V^{-1}UV)].\] 
These are all equal for $\|UV-VU\|$ small enough.

When the Bott invariant vanishes, it has been shown that there exist almost commuting unitaries (\cite{eilers1999morphisms}, \cite{gong1998almost}). However, no explicit estimate was known, as remarked in \cite{hastings2010almost}. This changed in 2025 when \cite{dor2025almost} resolved this issue.

\vspace{0.1in}

The proof of Hastings and Loring for constructing an exact representation of the sphere from an almost representation involves transforming almost representations of various geometric surfaces until one gets to an almost representation of the disk, which is then solved by Lin's theorem. The only non-trivial transformation of the respective spaces is the transformation of an almost representation of the sphere into an almost representation of the cylinder. This is done by utilizing the vanishing of the Bott index and a polar decomposition. The inverse transformation of ``pinching'' the boundary circles of a cylinder to form a sphere is not a diffeomorphism but is well-behaved. Each of these transformations can be done while respecting the norm of commutators although a square root is introduced.

In Section IV of \cite{hastings2010almost}, they discuss some of the difficulties with extending this approach to obtain an asymptotic estimate for the problem of two almost commuting unitaries. 
They note that there is not a nice geometric equivalence between the torus and the sphere that respects commutator norm estimates.

One can think of our method as noting that although a sphere and a torus really are not the same globally, there is a diffeomorphism between a half of the torus (alias ``slice'') such as $\{(e^{i\theta_1}, e^{i\theta_2}): 0< \theta_2 < \pi\}$ with a portion of the sphere with two poles cut off such as $\{(x,y,z): x^2+y^2+z^2, -1/2<x<1/2\}$. These are both a cylinder.

This geometric equivalence is used so we can make use of the relationship between the local projections of the unitary $V$ and the local projections for $H_1$ appropriately defined.

There is an index that must vanish in both cases, but when this index vanishes for the sphere, we can take local projections there for $H_1$ then use them for $V$. We apply the bootstrapping approach to the almost representation of the sphere $H_1 = f(V)$, $H_2 = \Re(g(V)+\frac12(Uh(V)+h(V)U))$, $H_3 = \Im(g(V)+\frac12(Uh(V)+h(V)U))$ and will use projections to obtain projections for $U, V$.

\vspace{0.1in}

Before continuing, we discuss how we will define the Bott index since there is some freedom in the choice of the functions $f, g, h$.
We define $h(e^{i\theta})$ as follows. $h(e^{i\theta})$ will be an operator Lipschitz function supported in $[e^{-i\pi/3}, e^{i\pi/3}]$. Specifically, we modify (3) of \cite{herrera2024projection} to obtain
\begin{align*}
h(x) = \left\{\begin{array}{ll} \sqrt{p(3x/\pi)}, & 0 \leq x \leq \frac\pi3\\
\sqrt{1-p(3x/\pi+1)}, & -\frac\pi3 \leq x \leq 0\\
0, & |x|>\frac\pi3 \\\end{array}\right.
\end{align*}
for $p(x)\geq 0$ a smooth strictly decreasing function on $[0,1]$ satisfying $p(0)=1$, $p(1)=0$, $p(1-x)=p(x)$ and with at least the first six derivatives of $p(x)$ at $x=0$ and at $x=1$ equal to zero. 

As in \cite{herrera2024projection}, $h(x)$ is even and $\sum_n h^2\left(x-\frac{\pi n}6\right) = 1$ on $\R$. Viewing $h$ as a $2\pi$-periodic function on $[-\pi, \pi]$, we see that translates of $h^2\left(x\right)$ form a partition of unity for the unit circle.

Let $g(x) = h(x+\pi)$. We view it as a periodic function on $[-\pi, \pi]$, so $g(x)$ is supported on $[-\pi, -2\pi/3]\cup [2\pi/3, \pi]$. Define 
\begin{align*}
f(x) = \left\{\begin{array}{ll} 
-\sqrt{1-g(x)^2}, & -\pi \leq x < -2\pi/3\\
-1, & -2\pi/3 \leq x < -\pi/3\\
-\sqrt{1-h(x)^2}, & -\pi/3 \leq x <0\\
\sqrt{1-h(x)^2}, & 0 \leq x < \pi/3\\
1, & \pi/3 \leq x < 2\pi/3\\
\sqrt{1-g(x)^2}, & 2\pi/3 \leq x < \pi
\end{array}\right.
\end{align*}
as a $2\pi$-periodic function that is supported on $[-\pi, \pi]$ with two roots: a root at 0 and a root at $-\pi=\pi$. See Figure \ref{fgh Functions}.
\begin{figure}[htp]     \centering
    \includegraphics[width=10cm]{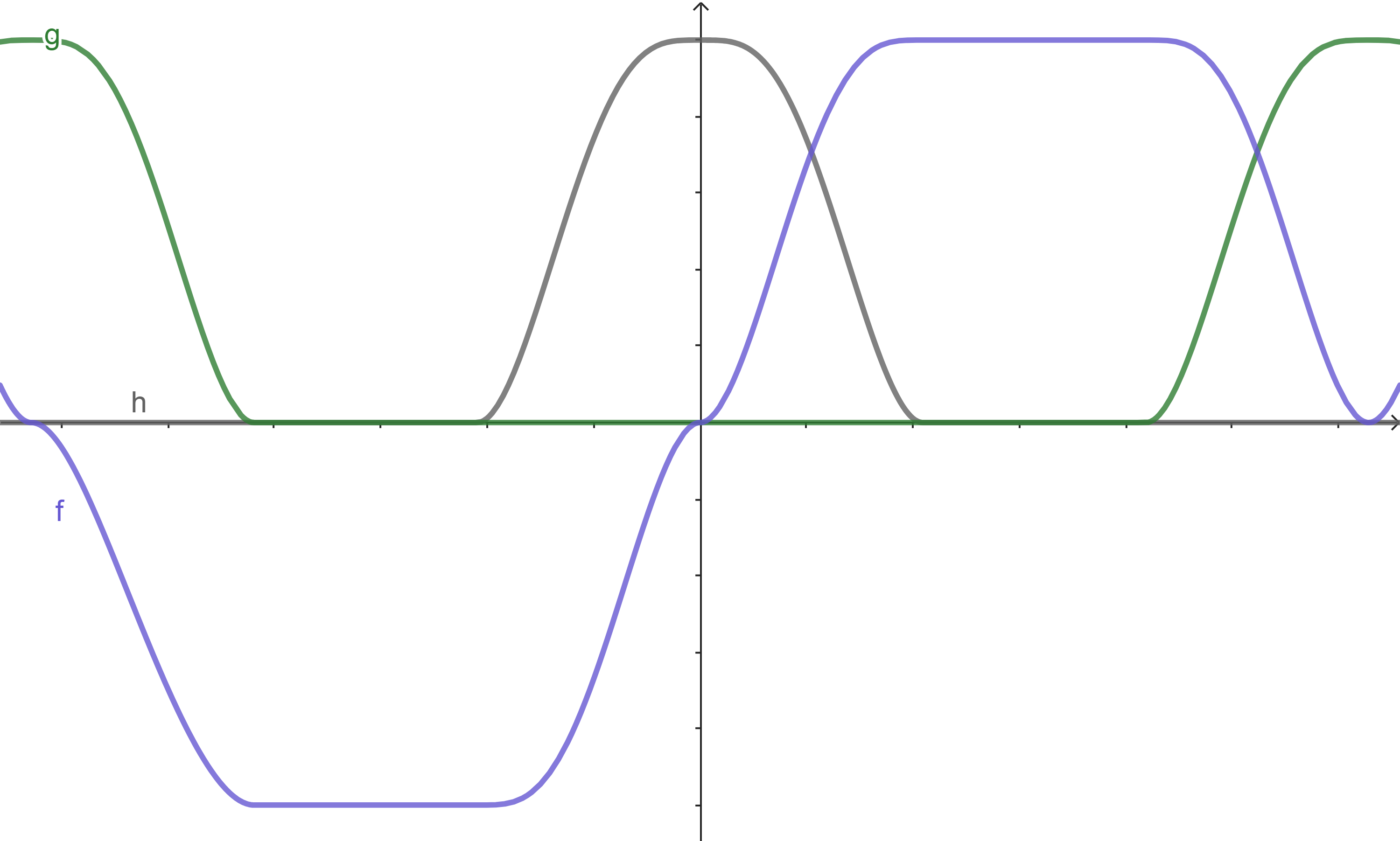}
    \caption{\label{fgh Functions}\dark
    Illustration of $f(x), g(x), h(x)$ on $[-\pi, \pi]$.}
\end{figure}

We will now view $f(x), g(x), h(x)$ as periodic functions on the unit circle.  By the choice of $p(x)$, one can assure that $f, g, h$ are thrice continuously differentiable. 
So, the Fourier coefficients of either of $f, g, h$ decay fact enough by \cite[Corollary 3.3.10]{loukas2014classical} so that by \cite[Theorem 1.1.3]{aleksandrov2016operator},
\[\|[f(V), U]\|, \|[g(V), U]\|, \|[h(V), U]\| \leq Const.\|[U,V]\|.\]

By construction $f^2+g^2+h^2=1$. Also $f$ is odd by construction because $h$ and $g$ are even. Note that if we restrict to $f$ to a compact subset of $(0, \pi/3)$ then $f^{-1}$ is Lipschitz on the respective image.
\begin{figure}[htp]     \centering
\begin{subfigure}[b]{0.3\textwidth}
         \centering
         \includegraphics[width=\textwidth]{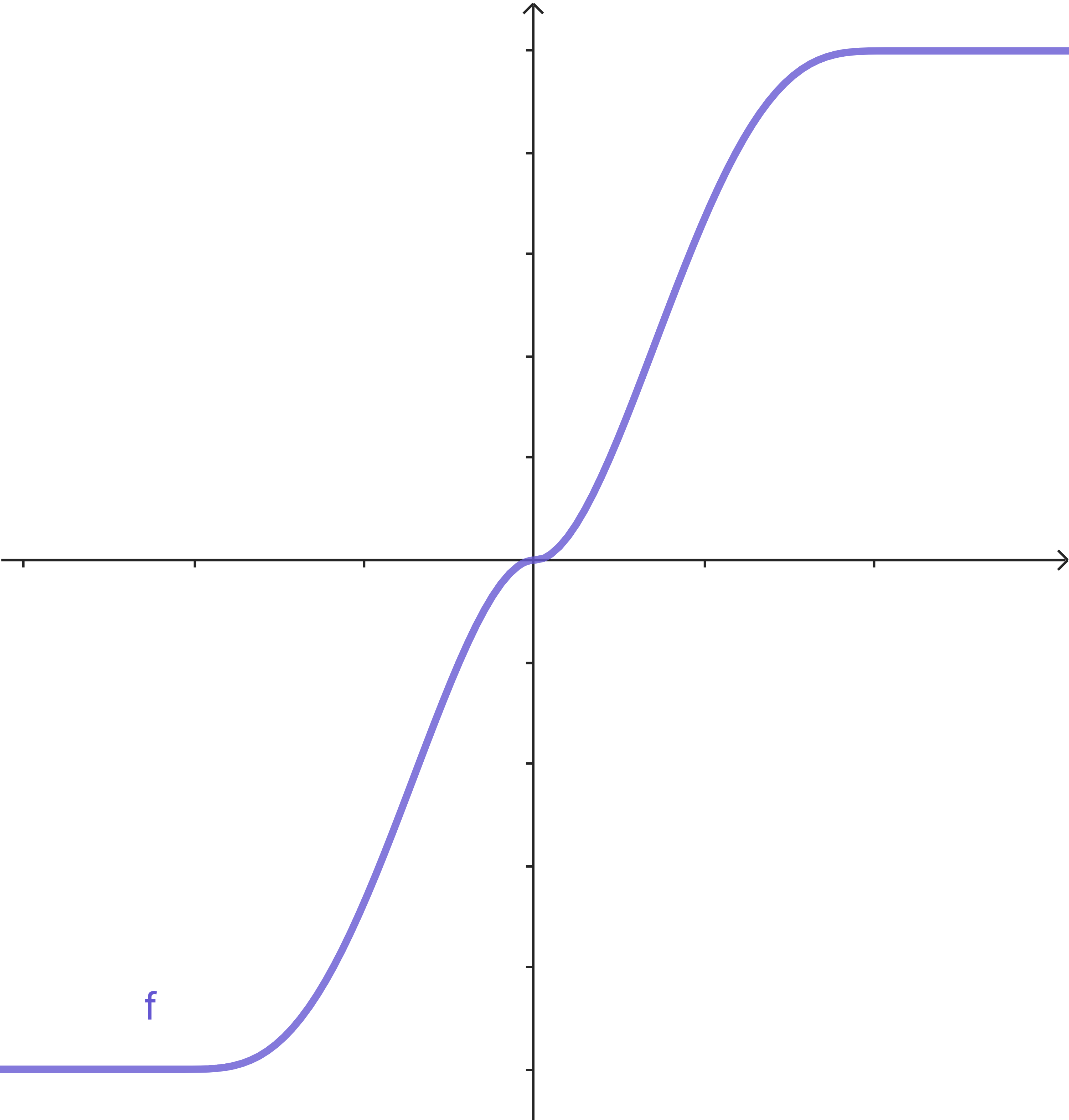}
         \caption{$f(x)$ on $[-\pi/2, \pi/2]$} \label{f Function}
     \end{subfigure}
     \hspace{1in}
     \begin{subfigure}[b]{0.3\textwidth}
         \centering
         \includegraphics[width=\textwidth]{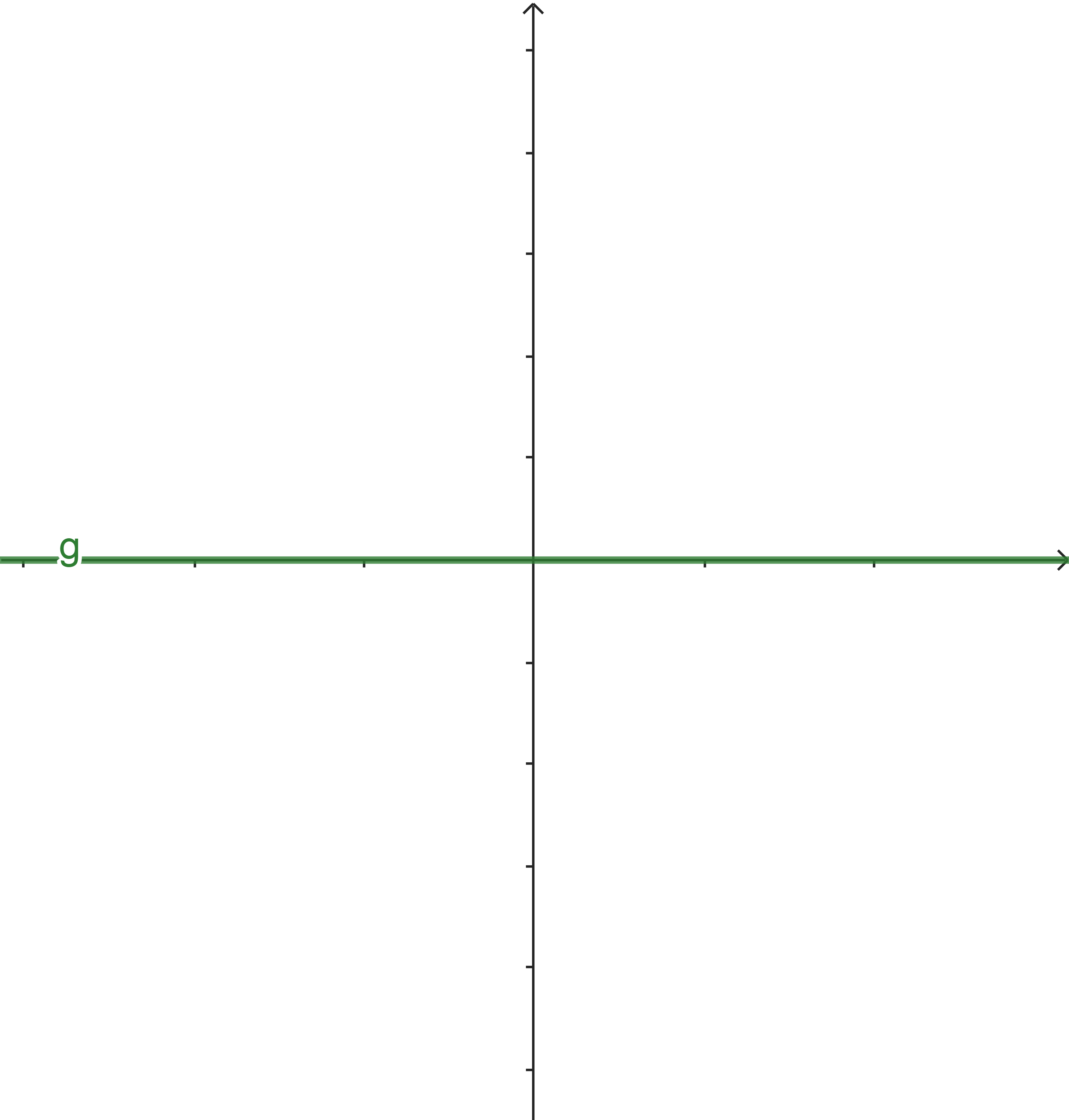}
         \caption{$g(x)$ on $[-\pi/2, \pi/2]$}\label{g Function}
     \end{subfigure}
     \hspace{1in}
     \begin{subfigure}[b]{0.3\textwidth}
         \centering
         \includegraphics[width=\textwidth]{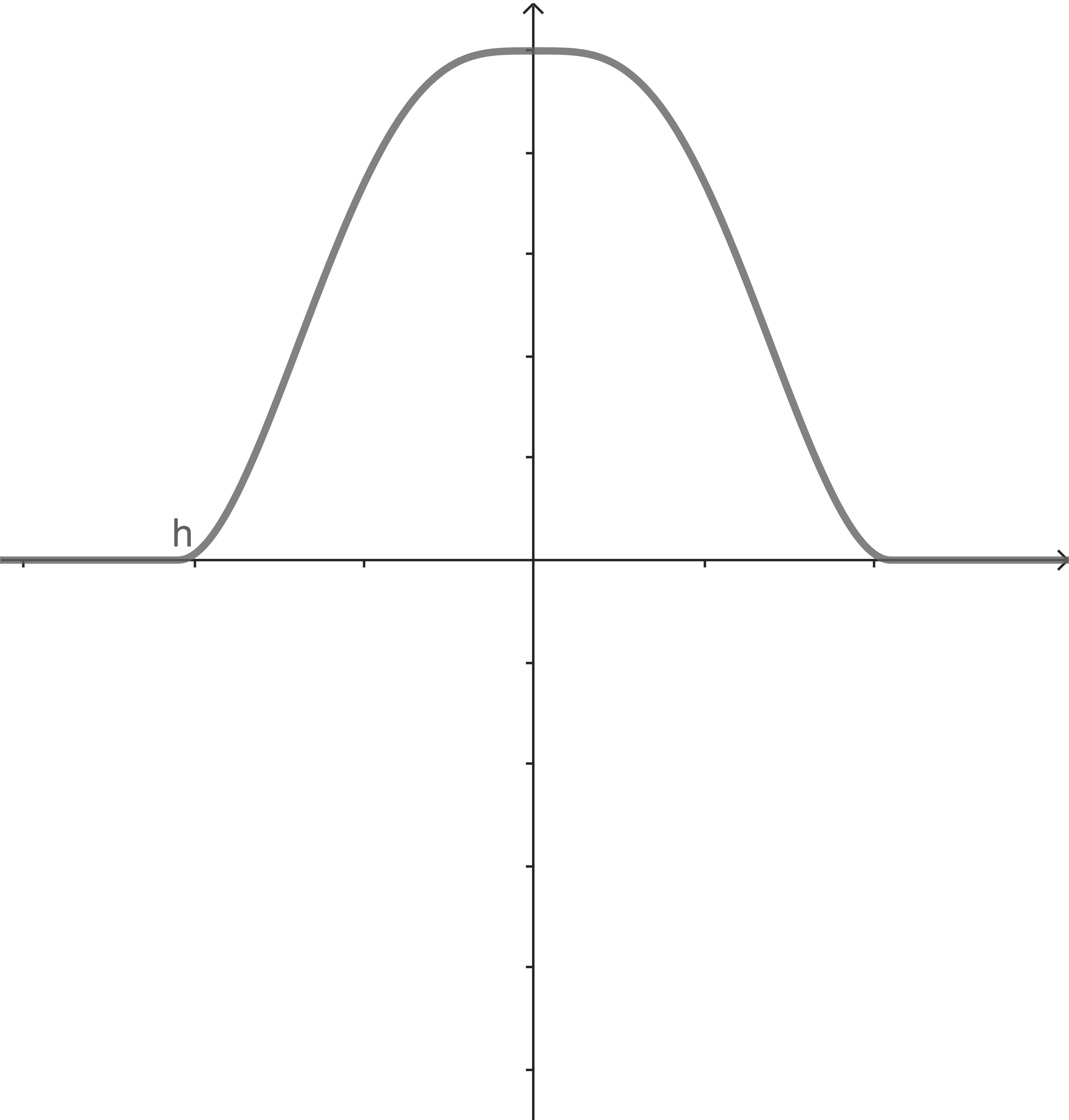}
         \caption{$h(x)$ on $[-\pi/2, \pi/2]$}\label{h Function}         
     \end{subfigure}
\end{figure}

Note that we do not really need to know the exact expressions of $f, g, h$. The main properties needed are the commutator estimates, their squares summing to $1$ and the qualitative properties reflected in the graph. Any smooth functions that have these properties would suffice. 

We now prove a lemma which states that we can localize the Bott index. 
This idea has already been used in proving the logarithm expression for the index in Lemma 2.2 of \cite{exel1991invariants}.

We want to localize the index for the part of the spectrum of $V$ near $0$. So, we will want to only look at the portion of the spectrum of $f(V)$ near $0$, however the problem is that $f(x)$ is zero near the origin (where we want to focus) and also near $\pm\pi$. So, the spectral projection of $f(V)$ on an open set containing zero will contain the desired spectral projection of $V$ but also a second part that we will discard using the following:
\begin{lemma}\label{index localization}
Let $U, V$ be unitary matrices in $M_n(\C)$ and let $E = E_{[e^{-i\pi/2},e^{i\pi/2}]}(V)$, where $[e^{-i\pi/2},e^{i\pi/2}] = \{e^{i\theta}: -\pi/2 \leq \theta \leq \pi/2\}.$

Then for $\|[U,V]\|$ small enough, the signature of $B(U,V)$ equals the signature of $(E\oplus 
 E)B(U,V)(E\oplus E)$.

Also, let $E_h=E_{[e^{-i\pi/3}, e^{i\pi/3}]}(V) \leq E$. Then
\[h(V)=E_hh(V)=h(V)E_h=Eh(V)=h(V)E.\]
Although neither $E$ nor $E_h$ necessarily almost commute with $U$, we have:
\begin{align}\label{U and E}
\|(1-E)UE_h\|, \|E_hU(1-E)\| \leq Const.\|[U,V]\|.
\end{align}
A similar inequality holds for any disjoint sets $K_1, K_2$ with $\dist(K_1, K_2)> 0$: 
\[\|E_{K_2}(V)UE_{K_1}(V)\| \leq Const_{K_1, K_2}\|[U,V]\|.\]
\end{lemma}
\begin{proof}
$E$ commutes with $f(V), h(V), g(V)$ and $Eh(V) = h(V)=E_hh(V)=h(V)E_h$, $Eg(V)=0$. 

Let \[B = B(U, V)= \bp f(V) & g(V)+\frac12\{U,h(V)\}\\
g(V)+\frac12\{U^\ast,h(V)\} & -f(V)
\ep.\]
Let $E_2 = E \oplus E$.
So, consider $B_E = E_2BE_2$ and $B_{1-E} = (1-E_2)B(1-E_2)$. Then 
\[\|B-B_E-B_{1-E}\| = \|(1-E_2)BE_2 + E_2B(1-E_2)\| \leq \|(1-E_2)BE_2\|.\]
We will estimate the norm of
\[(1-E_2)BE_2 = \frac12\bp 0 & (1-E)\{U,h(V)\}E\\
(1-E)\{U^\ast,h(V)\}E & 0
\ep.\]
Here we use that the support of $h(V)$ is $E_h=E_{[e^{-i\pi/3},e^{ i\pi/3}]}(V) \leq E$.
So, \[(1-E_2)BE_2 = \frac12\bp 0 & (1-E)U E_hh(V)\\
(1-E)U^\ast E_h h(V) & 0
\ep.\]

Note that $1-E$ and $E_h$ are spectral projections for $V$ for arcs in the unit circle that are separated by arcs of length $\pi/6$. We make use of the localization operator $\L^\Delta_{V}$ from Lemma 5.7 of \cite{herrera2024projection}. Choosing $\Delta$ small enough, we have by (iv) of that lemma, 
\[(1-E)\L(U)E_h = E_h\L(U)(1-E) = 0\] and by (iii) of that lemma,
\[\|\L(U) - U\| \leq Const.(\|[V, U]\|+ \|[V,U^\ast]\|) = Const.\|[U, V]\|.\]
We then get 
\[\|(1-E)UE_h\| , \|E_hU(1-E)\| \leq Const.\|[U, V]\|.\]
Similar bounds follow for general $K_1. K_2$ by choosing $\Delta$ small enough.

We then have
\[\|B-B_E-B_{1-E}\| \leq Const.\|[U, V]\|.\]
For $\|[U,V]\|$ small enough, $B_E+B_{1-E}$ will be invertible and have the same signature as $B$. By construction, $B_E+B_{1-E}$ is a block matrix with respect to $E, 1-E$ whose non-zero eigenvalues are the union of the non-zero eigenvalues of $B_E$ and $B_{1-E}$, with multiplicity. So, the signature of $B_E+B_{1-E}$ equals the sum of the signatures of $B_E$ and of $B_{1-E}$.

We now investigate the signature of $B_{1-E}$. This is straightforward because $h(V)(1-E)=0$ so
\[B_{1-E}= \bp f(V)(1-E) & g(V)(1-E)\\
g(V)(1-E) & -f(V)(1-E)
\ep.\]
By the spectral theorem, $B_{1-E}$ is similar to the direct sum of a zero matrix and $2\times 2$ matrices of the form
\[\bp f(w_0) & g(w_0)\\ g(w_0) & -f(w_0)\ep\]
for $w_0 \in [e^{-i\pi},e^{i\pi}] \setminus [e^{-i\pi/2}, e^{i\pi/2}]$. These $2 \times 2$ matrices have trace equal to zero so their signature is zero. Thus, the signature of $B_{1-E}$ is zero.

This tells us that the signature of $B$ is equal to the signature of $B_E$ for $\|[U,V]\|$ small.
\end{proof}

\begin{example}\label{BottLocalizationEx}
We know that in order for $U, V$ to have nearby commuting unitaries, it is required that the signature of $B(U, V)$ vanish. We will assume this.

Note that we can continuously change $V_0$ to $V$ by multiplying by $e^{i\theta}$ with $\theta$ increasing from $0$ without affecting the norm of the commutator of $U, V_0$ and without $B(U,V_0)$ becoming singular. This is equivalent to continuously shifting the functions $f, g, h$. In particular, the signature will remain zero after replacing $V_0$ with $V$.

Let $E$ be the spectral projection of $V$ for the closed arc from $-\pi/2$ to $\pi/2$. 
By Lemma \ref{index localization}, the signature of $B_E = EB(U,V)E$ is also zero. We will pull out of this an almost representation of the sphere.

Note that 
\[B_E = \bp f(V)E & \frac12E\{U,h(V)\}E\\
\frac12E\{U^\ast, h(V)\}E & -f(V)E
\ep.\]
We define 
\[H_1 = f(V)E,\;\; H_2 = \frac12\Re\left(E\{U,h(V)\}E\right),\;\; H_3 = \frac12\Im\left(E\{U,h(V)\}E\right).\] 
We now show that $H_1, H_2, H_3$ form an almost representation of the sphere as operators on the finite dimensional Hilbert space $R(E)$ which has identity operator $E$.

Clearly, $-E \leq H_i \leq E$. We will provide bounds for the commutators. Note that because  $i[H_1,H_2]$, $i[H_1,H_3]$ are self-adjoint, $\|[H_1, H_2]\|, \|[H_1, H_3]\| \leq \|[H_1, H_2+iH_3]\|$. So, we bound
\begin{align*}
\|[H_1, H_2+iH_3]\| &= \|[f(V)E, 
 \frac12 E\{U, h(V)\}E]\| \\
 &= \frac12\|E[f(V),U]h(V)E+Eh(V)[f(V),U]E\| \\
 &\leq \|[f(V), U]\| \leq Const. \|[U,V]\|. 
\end{align*}

By (\ref{U and E}), $\|E_hUE - E_hU\|, \|EUE_h - UE_h\| \leq Const.\|[U,V]\|$.
Now we estimate
\begin{align*}
\|[\,E\{U^\ast&,h(V)\}E, E\{U,h(V)\}E\,]\| \\
&= \|\,[EU^\ast E_hh(V)+h(V)E_hU^\ast E, EUE_hh(V)+h(V)E_hUE\,]\|.
\end{align*}
This breaks into four terms:
\begin{align*}
\|[EU^\ast E_hh(V), EUE_hh(V)]\|&= \|E\left(U^\ast h(V)U - Uh(V)U^\ast \right)h(V)\|\\
&\leq \|U^\ast h(V)U - Uh(V)U^\ast\| =\|h(V)U^2 - U^2h(V)\| \\
&\leq 2\|[h(V), U]\| \leq Const.\|[U,V]\|,
 \end{align*}
\begin{align*}
\|[EU^\ast & E_hh(V),h(V)E_hUE]\| = \|E\left(U^\ast h^2(V)U-h(V)E_hUEU^\ast h(V)\right)E\| \\
&\leq \|h^2(V)-h(V)(E_hU)U^\ast h(V)\| + Const.\|[U,V]\|=Const.\|[U,V]\|,
\end{align*}
\begin{align*}
\|[h(V)&E_hU^\ast E, EUE_hh(V)]\|=\|E\left(h(V)U^\ast EUE_hh(V)-Uh^2(V)U^\ast\right)E\|\\
&\leq \|h(V)U^\ast (UE_h)h(V)-h^2(V)\| + Const.\|[U,V]\|=Const.\|[U,V]\|,
\end{align*}
\begin{align*}
\|[h(V)E_hU^\ast E, h(V)E_hUE]\|&=\|h(V)\left(U^\ast  h(V)U-Uh(V)U^\ast\right)E\|\\
&\leq Const.\|[U,V]\|.
\end{align*}

So,
\begin{align*}
\|[H_2, H_3]\| &= \frac12\|[(H_2+iH_3)^\ast, H_2+iH_3]\| \\
&\leq \left\|\left[E\{U^\ast,h(V)\}E, E\{U,h(V)\}E\right]\right\| \leq Const.\|[U,V]\|.
\end{align*}

So, $H_1, H_2, H_3$ are almost commuting. They satisfy 
\begin{align*}
\|&H_1^2 + H_2^2+H_3^2-E\| \\
&\leq \|H_1^2 + (H_2+iH_3)^\ast(H_2+iH_3)-E\| + Const.\|[U,V]\|\\
&= \|E\left(f^2(V)+\frac14\{U^\ast,h(V)\}E\{U,h(V)\}-I\right)E\|+Const.\|[U,V]\|\\
&\leq \|E\left(f^2(V)+h^2(V)-I\right)E\|+Const.\|[U,V]\|=Const.\|[U,V]\|
\end{align*}
as above.

So, $H_1, H_2, H_3$ form an almost representation of the sphere in the Hilbert space $R(E)$. Because $B(H_1, H_2, H_3) = B_E(U,V)$ has signature zero, there are commuting self-adjoint $H_i'$ with ${H_1'}^2+{H_2'}^2+{H_3'}^2 = E$ and $\|H_i' - H_i\| \leq Const. \|[U,V]\|^{1/4}$.
\end{example}

To discuss the symmetries that we can preserve or obtain, we make use of the definition of a symmetry map from \cite[Definition 2.2]{herrera2024projection}:
\begin{defn} An $\R$-linear map $\varphi$ on a unital $C^\ast$-algebra $\mathcal A$ is a symmetry map if:
\begin{enumerate}
\item $\varphi$ is $\C$-linear or conjugate linear.
\item $\varphi$ is multiplicative or anti-multiplicative.
\item $\varphi(I)=I$ and $\varphi(A^\ast)=A^\ast$ for all $A \in \mathcal A$.
\end{enumerate}

\end{defn}

We now prove:
\begin{lemma}\label{UtoHtoU}
Let $\mathcal A$ be a von Neumann algebra and $S$ a collection of weakly continuous linear symmetry maps on $\mathcal A$. Let $U, V \in \mathcal A$ be $S$-symmetric unitaries.
For $z \in \bS$, we define $E$ to be the spectral projection of $V$ on the arc $(e^{-\pi i/2}z, e^{\pi i/2}z)\subset \bS$ and the $S$-symmetric self-adjoint $H_1, H_2, H_3 \in \mathcal A$ by \[H_1 = f(\overline{z}V)E,\;\; H_2 = \frac12\Re\left(E\{U,h(\overline{z}V)\}E\right),\;\; H_3 = \frac12\Im\left(E\{U,h(\overline{z}V)\}E\right).\]
Suppose that there exists an $\varepsilon>0$ so that for any $z \in \bS$, there are commuting $S$-symmetric self-adjoint $H_i'\in \mathcal A$ for the defined $H_i$ operators above with:
\[\|H_1'-H_1\|, \|H_2'-H_2\|, \|H_3'-H_3\| \leq \varepsilon.\] 
Then if $\varepsilon$ is universally small enough, there are commuting $S$-symmetric unitaries $U', V' \in \mathcal A$ such that
\[\|V'-V\|\leq C\varepsilon^{1/2},\;\; \|U'-U\|\leq C\varepsilon^{1/2}+C\frac{\|[U,V]\|}{\varepsilon^{1/2}}.\]
for a universal constant $C$.
Moreover, if $U, V$ have a phase symmetry as in one of the following cases then $U', V'$ can be chosen to also satisfy the same phase symmetry:
\begin{enumerate}[label=(\roman*)]
\item Suppose that $\phi$ is a weakly continuous conjugate-linear symmetry map on $\mathcal A$ with order $2$ such that $\phi(U) = \zeta_0 U$ and $\phi(V) = \eta_0 V$ (resp. $\phi_\ast(V) = \eta_0 V$). Then $U', V'$ can be chosen such that $\phi(U') = \zeta_0 U'$ and $\phi(V') = \eta_0 V'$ (resp. $\phi_\ast(V') = \eta_0 V'$).

\item Suppose that $\varphi$ is a weakly continuous linear symmetry map on $\mathcal A$ with order $n$ such that $\varphi(U) = \zeta_1 U$ for $\zeta_1$ having order $n$ and $\varphi(V) = \eta_1 V$ (resp. $\varphi_\ast(V) = \eta_1 V$). Then $U', V'$ can be chosen such that $\varphi(U') = \zeta_1 U'$ and $\varphi(V') = \eta_1 V'$ (resp. $\varphi_\ast(V') = \eta_1 V'$).

In this case, $C$ will depend on $n$.

\item Suppose that $\varphi, \phi$ are as in Cases (i) and (ii) above and satisfy $\phi \circ \varphi \circ \phi = \varphi^{-1}$. Then $U', V'$ can be chosen in a way that they satisfy the four phase symmetries of (i) and (ii). 

In this case, $C$ will depend on $n$.
\end{enumerate}
\end{lemma}
\begin{proof}
We will apply Lemmas 6.1, 6.4, and 6.6 of \cite{herrera2024projection} analogous to the proof of Theorem 7.6 of \cite{herrera2024projection}.

Let the point $z_0$ in the unit circle be arbitrary. Let $V_z=\overline{e^{-\pi i/6}z_0}V$ so that $E=E_{(e^{-\pi i/2}, e^{\pi i/2})}(V_z)$ is the spectral projection of $V$ on the open semicircle centered at $e^{-\pi i/6}z_0$. 

Now, $E_{(0,1)}(f(V_z)E) = E_{(0, \pi/3)}(V_z)$. Moreover, for $0< y_1< y_2 < 1$, we have
$E_{(y_1, y_2)}(f(V_z   )E) = E_{(f^{-1}(y_1),f^{-1}(y_2))}(V_z)$. 
We know that $f$ and $f^{-1}$ are bijective and continuous on the relevant intervals $(e^{-\pi i/3}, e^{\pi i/3})$ and $(-1,1)$, respectively. Restricted to $(e^{\pi i/12}, e^{\pi i /4})$, $f$ and its inverse on the image are Lipschitz.
Note that we have chosen an arc centered at $z_0$ for $V$ to correspond to an arc centered at $e^{\pi i/6}$ for $V_z$ so that $f^{-1}$ will be Lipschitz on this latter interval if it has length less than $2(\pi/6)$.

Choose arcs $\Omega_0 \subset \subset \Omega \subset (e^{\pi i/12}, e^{\pi i/4})\subset \bS$ centered at $e^{ \pi i/6}$ with left and right subarcs $\Omega_-, \Omega_+$ of $\Omega\setminus \Omega_0$. Let $\Omega_-,\Omega_0, \Omega_+$ have the same length.
Let $L_0$ be small enough so that $\dist(\bS\setminus \Omega, \Omega_0)$ and $\dist(\Omega_-, \Omega_+)$ will be between $L$ and $2L$ for $L \leq L_0$. These conditions imply that the length and diameter of $\Omega$ is bounded by $Const.L$.

We will need to construct $S$-symmetric projections $F, F_-, F_+$ such that
$F=F_-+E_{\Omega_0}(V_z)+F_+$ where $F_- \leq E_{\Omega_-}(V_z)$, $F_+ \leq E_{\Omega_+}(V_z)$, and 
\[\|[F, U]\| \leq \epsilon\]
for some expression $\epsilon$ that we will determine later. 

Consider the $S$-symmetric operators $X = H_1$ and $B = H_2+iH_3$, where $X$ is self-adjoint. By assumption, there are commuting $S$-symmetric operators $X'=H_1'$ and $B' = H_2'+iH_3'$, where $X'$ is self-adjoint and
\[\|X'-X\|, \|B'-B\| \leq 2\varepsilon.\]
So, by Lemma 6.8 of \cite{herrera2024projection}, there are $S$-symmetric projections 
$F$, $F_-, F_+$ so that 
\[F = F_-+E_{f(\Omega_0)}(f(V_z)E)+F_+=F_-+E_{\Omega_0}(V_z)+F_+,\] 
\[F_- \leq E_{f(\Omega_-)}(f(V_z)E)=E_{\Omega_-}(V_z),\] 
\[F_+ \leq E_{f(\Omega_+)}(f(V_z)E)=E_{\Omega_+}(V_z),\] 
and
\[\|[F, B]\| \leq Const.\left(\frac1{\dist(\R\setminus f(\Omega), f(\Omega_0))}+\frac1{\dist(f(\Omega_-), f(\Omega_+))}\right)\varepsilon.\]
Because $f^{-1}$ is Lipszhitz on $(f(e^{\pi i/12}), f(e^{\pi \pi/4}))$, we see that
\[\|[F, B]\|\leq Const.\left(\frac1{\dist(\bS\setminus \Omega, \Omega_0)}+\frac1{\dist(\Omega_-, \Omega_+)}\right)\varepsilon.\]

We now see what this inequality tells us about $[F, U]$. Roughly speaking, 
\[FB \approx h_\Omega FU, \;\;\; BF \approx h_\Omega UF\]
for $h_\Omega \in h(\Omega)$ since $F\leq E_{\Omega}(V_z)$. The error in the ``$\approx$'' will depend on the length of $\Omega$. This will give us a bound for $\|[F,U]\|$ in terms of the above bound for $\|[F,B]\|$ because $h_\Omega$ wil be bounded away from zero.

Let $\ell$ be the length of $\Omega$, which is bounded by $Const. L$. Let $h_\Omega$ be the midpoint of $h(\Omega)$. Then \[\|(h(V_z)-h_{\Omega}I)E_{\Omega}(V_z)\| \leq \frac12\|h'\|_{L^\infty}\ell = Const. \ell.\]
Let $\Omega_U$ be an arc centered at $1$ of slightly greater length by $Const.L$ than $\Omega$. By Lemma \ref{index localization}, 
\[\|E_{\Omega}(V_z)U-E_{\Omega}(V_z)UE_{\Omega_U}(V_z)\|, \|UE_{\Omega}(V_z)-E_{\Omega_U}(V_z)UE_{\Omega}(V_z)\| \leq Const.\|[U,V]\|.\]
Since $F \leq E_{\Omega}(V_z)$, this implies
\begin{equation}\label{UF restr}
\|FU-FUE_{\Omega_U}(V_z)\|, \|UF-E_{\Omega_U}(V_z)UF\| \leq Const.\|[U,V]\|.
\end{equation}
So,
\begin{align*}
h_{\Omega}&\|[F,U]\| = \|h_{\Omega}FU-h_{\Omega}UF\|\\
&\leq \|h_{\Omega}FU-h_{\Omega}UF-F(H_2+iH_3)+(H_2+iH_3)F\|+\|[F,H_2+iH_3]\|\\
&\leq \|h_{\Omega}FU-F(H_2+iH_3)\|+\|h_{\Omega}UF-(H_2+iH_3)F\|+\|[F,B]\|
\end{align*}
Using $F \leq E_{\Omega}(V_z)\leq E_{\Omega_U}(V_z) \leq E$ and (\ref{UF restr}), we bound the first two terms of the above inequality in four parts:
\begin{align*}
\|h_{\Omega}FU-F&EUh(V_z)E\| \\\leq &\|h_{\Omega}FUE_{\Omega_U}(V_z)-FUE_{\Omega_U}(V_z)h(V_z)\| + Const.\|[U,V]\|\\
&\leq\|FUE_{\Omega_U}(V_z)(h_{\Omega}I-h(V_z))\| + Const.\|[U,V]\|\\
&\leq Const.\ell+Const.\|[U,V]\|,
\end{align*}
\begin{align*}
\|h_{\Omega}FU-FEh(V_z)UE\|&\leq 
\|h_{\Omega}FU-FEUh(V_z)E\| + Const.\|[U,V]\|\\
&\leq Const.\ell+Const.\|[U,V]\|,
\end{align*}
\begin{align*}
\|h_{\Omega}UF&-Eh(V_z)UEF\|\\
&\leq \|h_{\Omega}E_{\Omega_U}(V_z)UF-h(V_z)E_{\Omega_U}(V_z)UF\|+Const.\|[U,V]\|\\
&= \|(h_{\Omega}I-h(V_z))E_{\Omega_U}(V_z)UF\|+Const.\|[U,V]\|\\
&\leq Const.\ell+Const.\|[U,V]\|.
\end{align*}
\begin{align*}
\|h_{\Omega}UF-EUh(V_z)EF]\|&\leq \|h_{\Omega}UF-Eh(V_z)UEF\|+Const.\|[U,V]\|\\
&\leq Const.\ell+Const.\|[U,V]\|.
\end{align*}

If we use the fact that $\dist(\bS\setminus \Omega, \Omega_0)$, $\dist(\Omega_-, \Omega_+), \ell$ are comparable to $L$ then
\begin{align}\label{[F,U]}
\|[F,U]\|&\leq  \frac1{h_\Omega}\left(\|[F,B]\|+Const.\ell+Const.\|[U,V]\|\right)\\
&\leq Const.\left(\frac\varepsilon{L}+L+\|[U,V]\|\right)=:\epsilon(L).
\end{align}
Note that the constant $Const.$ in the definition of $\epsilon(L)$ is a universal constant.

So, we have the desired projections for $V_z$ for small enough arcs $\Omega_0, \Omega$ in $(e^{\pi i/12}, e^{\pi i/4})$ centered at $e^{\pi i/6}$. This then gives projections for $V$ with corresponding arcs $\tilde \Omega_0, \tilde \Omega$ centered at any $z_0$. 
This gives the desired projections since $\epsilon(L)$ and how small the arcs need to be do not depend on $z_0$. So, these satisfy (1)-(3) of Lemma 6.4 in \cite{herrera2024projection}.

Hence, by applying Lemma 6.4 and 6.6 of \cite{herrera2024projection} with $\tilde U = V, \tilde B = U$, we have proved the existence of commuting unitaries $U', V'$ with the desired structural properties, noting that the estimates in the statement of our lemma follow from (15) of \cite{herrera2024projection}:
\[\|V'-V\|\leq Const.L,\;\; \|U'-U\|\leq \frac{Const.}L\|[V,U]\|+2\epsilon(L),\]
\[\|V'-V\|\leq Const.L,\;\; \|U'-U\|\leq Const.\left(\frac1L\|[U,V]\|+\frac\varepsilon{L}+L\right).\]
Because we expect $\varepsilon > \|[U,V]\|$, we choose $L = \varepsilon^{1/2}$ to obtain the estimate in the statement of the lemma.

Note that the requirement that $L < L_0$ in Lemma 6.4 of \cite{herrera2024projection} will be satisfied because $L_0$ only depends the structure of the symmetries (e.g. $n$) and needs to be small enough that for any $L<L_0$, the estimate in (\ref{[F,U]}) holds. So, $\varepsilon$ can be chosen to be universally small enough so the construction always works.
\end{proof}
\begin{remark}
Because the bootstrap method worsens the asymptotic behavior of an estimate (usually by multiplying the exponent by $1/2$), we suspect that this method is likely not able to produce an asymptotically optimal result for almost commuting unitary matrices. So, we will not try to dig into the black box of the treatment of the almost representation of the sphere in \cite{hastings2010almost} to attempt to improve our above estimate.
\end{remark}

Our work in Example \ref{BottLocalizationEx} together with Lemma \ref{UtoHtoU} proves the following theorem.
\begin{thm}\label{UtoHtoUThm}
Let $U, V \in M_d(\C)$ be unitaries with $B(U,V)$ having signature zero.
Then there are commuting unitaries $U', V' \in M_d(\C)$ such that
\[\|U'-U\|, \|V'-V\| \leq Const.\|[U,V]\|^{1/8}.\]
for a universal constant $Const$ that does not depend on $d$.
Moreover, if $U, V$ have a phase symmetry as in one of the following cases then $U', V'$ can be chosen to also satisfy the same phase symmetry:
\begin{enumerate}[label=(\roman*)]
\item Suppose that $\phi$ is a weakly continuous conjugate-linear symmetry map on $\mathcal A$ with order $2$ such that $\phi(U) = \zeta_0 U$ and $\phi(V) = \eta_0 V$ (resp. $\phi_\ast(V) = \eta_0 V$). Then $U', V'$ can be chosen such that $\phi(U') = \zeta_0 U'$ and $\phi(V') = \eta_0 V'$ (resp. $\phi_\ast(V') = \eta_0 V'$).

\item Suppose that $\varphi$ is a weakly continuous linear symmetry map on $\mathcal A$ with order $n$ such that $\varphi(U) = \zeta_1 U$ for $\zeta_1$ having order $n$ and $\varphi(V) = \eta_1 V$ (resp. $\varphi_\ast(V) = \eta_1 V$). Then $U', V'$ can be chosen such that $\varphi(U') = \zeta_1 U'$ and $\varphi(V') = \eta_1 V'$ (resp. $\varphi_\ast(V') = \eta_1 V'$).

In this case, $Const.$ will depend on $n$.

\item Suppose that $\varphi, \phi$ are as in Cases (i) and (ii) above and satisfy $\phi \circ \varphi \circ \phi = \varphi^{-1}$. Then $U', V'$ can be chosen in a way that they satisfy the four phase symmetries of (i) and (ii). 

In this case, $Const.$ will depend on $n$.
\end{enumerate}
\end{thm}
\begin{remark}
Theorem \ref{unitary thm} stated in the introduction is a restatement of this theorem with the symmetries of (i) or (ii) because any symmetry map on $M_d(\C)$ can be written as the composition of conjugation by a unitary, taking the transpose, and taking the adjoint.  See Proposition 2.6 of \cite{herrera2024projection}.

The symmetry in case (iii) is a dihedral symmetry.
\end{remark}
\begin{remark}
This provides an asymptotic estimate for almost commuting unitary matrices for which the Bott index vanishes as well as real almost commuting unitaries and symplectic almost commuting unitaries. 

If we wish to extend our result to prove Lin's theorem for transpose-symmetric unitaries or self-dual unitaries (or anti-multiplicative symmetries more generally), then we would need to have an asymptotic result for the almost representation of the sphere with these linear symmetries. It appears plausible that one can derive these from the decompositions of \cite{loring2013almost}, which would provide asymptotic estimates for the results of \cite{loring2013almost}. 

Then our Theorem \ref{UtoHtoUThm} would immediately extend to the anti-multiplicative case because the only bottleneck for linear symmetries is the the passage from $H_i$ to $H_i'$.
\end{remark}

\begin{example}\label{Linear Symmetry Fail for Unitaries}
Although we can obtain conjugate-linear symmetries or phase symmetries using our symmetry bootstrap method for almost commuting unitaries, we cannot for multiplicative linear symmetries. In fact, there are some interesting examples to consider. 

For instance, Loring in Proposition 6.1 of \cite{loring1988k} showed that, one can let $n$ be a square and then the variant of Voiculescu's unitaries: $U_n\oplus U_n, V_n \oplus V_n^\ast$ are nearby commuting $U_n', V_n'\in M_{2n}(\C)$ with \[\|U_n'- U_n\oplus U_n\|, \|V_n'- V_n\oplus V_n^\ast\| \leq 4\pi \|[U_n, V_n]\|^{1/2}.\]
One can show that the Bott index vanishes for these direct sum unitaries, however one should notice that although both $U_n\oplus U_n, V_n\oplus V_n^\ast$ are block-diagonal, $U_n', V_n'$ cannot be. This is because the blocks would then be commuting unitaries that are nearby Voiculescu's unitaries.

We see from this example that if we consider the symmetry $\varphi(A) = \Gamma^{-1}A\Gamma$, where $\Gamma = I \oplus (-I)$, then then being block diagonal is equivalent to being $\varphi$-symmetric. So, both $U_n\oplus U_n, V_n\oplus V_n^\ast$ are $\varphi$-symmetric, but there are no nearby commuting unitaries that are also $\varphi$-symmetric. This is despite that fact that the signature of $B(U_n\oplus U_n, V_n\oplus V_n^\ast)$ vanishes so there is \emph{some} pair of nearby commuting unitaries.
\end{example}

\begin{example}
Another notable example we provide is when we have two almost commuting unitaries that satisfy a phase symmetry but the phase does not have the same order as the symmetry. More precisely, suppose we are considering almost commuting unitaries $U, V$ on $M_{4n}(\C)$. Let $\Gamma = (iJ) \oplus J$, where $J = \bp 0_n & 
I_n\\ -I_n & 0_n\ep$. Define $\varphi(A) = \Gamma^{-1}A\Gamma$ and we assume that $U, V$ are $\varphi$-antisymmetric.

Note that $J^2 = -1$, so $\Gamma^2 = I \oplus -I$. So, if $U$ is $\varphi$-antisymmetric, it is automatically symmetric with respect to conjugation by $\Gamma^2$. This ``hidden'' symmetry is due to the fact that the phase $\zeta = -1$ has order $2$ but conjugation by $\Gamma$ has order $4$.

Writing $U = \bp A&B \\ C&D\ep$ with $A, B, C, D \in M_{2n}(\C)$, one sees that $\varphi(U) = -U$ implies that $B = C = 0$ by applying $\varphi$ twice. So, $U = \bp A&0 \\ 0&D\ep$ and $J^{-1}AJ  = -A$, $J^{-1}DJ  = -D$. We see that $A, D$ have the structure  $\bp X & Y \\ Y & -X \ep$ for $X, Y \in M_n(\C)$.

So, consider the following pair of almost commuting unitaries which is a variant of the previous example:
\[U = 
\bp U_n &0_n&&\\
     0_n&-U_n&&\\
    &&U_n&0_n\\
    &&0_n&-U_n
     \ep, \;
     V = 
\bp V_n &0_n&&\\
     0_n&-V_n&&\\
    &&V_n^\ast&0_n\\
    &&0_n&-V_n^\ast
     \ep.\]
We view these as $2\times 2$ block matrices with blocks in $M_{2n}(\C)$. The Bott invariants for the first and for the second blocks separately do not vanish, but the index for $U, V$ vanishes. 
So, there are nearby commuting unitaries $U', V'$, but they cannot be chosen to be symmetric with respect to conjugation by $\Gamma^2$. So, there are no nearby commuting unitaries $U', V'$ that are $\varphi$-antisymmetric.

This example shows the importance of the phase having the same order as the symmetry for almost commuting unitary matrices in (ii) of Lemma \ref{UtoHtoU} and analogous theorems.
\end{example}

\section{Almost Anti-commuting Self-Adjoint Operators}

\begin{example}\label{2x2structure}
Let $\mathcal N = \bp A & B \\ C & D\ep$,  $\mathcal W = \bp E & F \\ G & H \ep \in M_{2}(\mathcal A)$ with $\mathcal W$ unitary. Then
\begin{align*}
{\mathcal W}^{-1}&{\mathcal N}{\mathcal W} = \bp E^\ast & G^\ast \\ F^\ast & H^\ast \ep\bp AE+BG & AF+BH \\ CE+DG & CF+DH\ep \\
&= \bp 
E^\ast A E+E^\ast BG+G^\ast CE+G^\ast DG&
E^\ast AF+E^\ast BH+G^\ast CF+G^\ast DH\\
F^\ast AE+F^\ast BG+H^\ast CE+H^\ast DG&
F^\ast AF+F^\ast BH+H^\ast CF+H^\ast DH\ep.
\end{align*}

We now list several examples of $\mathcal W$ and the effect of conjugating:
\begin{enumerate}[label=(\roman*)]
\item If $\mathcal W = \bp I & 0 \\ 0 & \alpha I\ep$ with $\alpha \in \C$ with $|\alpha|=1, \alpha \neq 1$ then
\[{\mathcal W}^{-1}{\mathcal N}{\mathcal W} = \bp 
A&
\alpha B\\
\overline{\alpha}C&
D\ep.
\]
So,
\[{\mathcal W}^{-1}{\mathcal N}{\mathcal W} = \mathcal N \;\; \Leftrightarrow\;\; B = C = 0 \;\; \Leftrightarrow\;\; \mathcal N = \bp A & 0 \\ 0 & D\ep.
\]
If $\alpha = -1$ then
\[{\mathcal W}^{-1}{\mathcal N}{\mathcal W} = -\mathcal N \;\; \Leftrightarrow\;\; A = D = 0 \;\; \Leftrightarrow\;\; \mathcal N = \bp 0 & B \\ C & 0\ep.
\]

\item If $\mathcal W = \bp 0 & \alpha I \\ I & 0\ep$ with $\alpha \in \C$ with $|\alpha|=1$ then
\[{\mathcal W}^{-1}{\mathcal N}{\mathcal W} = \bp 
D&
\alpha C\\
\overline{\alpha} B&
A\ep.
\]
So,
\[{\mathcal W}^{-1}{\mathcal N}{\mathcal W} = \mathcal N \;\; \Leftrightarrow\;\; D = A, C = \overline{\alpha} B \;\; \Leftrightarrow\;\; \mathcal N = \bp A & B \\ \overline{\alpha}B & A\ep.
\]
\[{\mathcal W}^{-1}{\mathcal N}{\mathcal W} = -\mathcal N \;\; \Leftrightarrow\;\; D = -A, C = -\overline{\alpha} B \;\; \Leftrightarrow\;\; \mathcal N = \bp A & B \\ -\overline{\alpha}B & -A\ep.
\]

\item If $\mathcal W = \frac{\sqrt2}2\bp I & -I \\ I  & I\ep$ then 
\[{\mathcal W}^{-1}{\mathcal N}{\mathcal W} = \frac12\bp 
A+B+C+D&
-A+B-C+D\\
-A-B+C+D&
A-B-C+D\ep.
\]
So, 
\[{\mathcal W}^{-1}{\mathcal N}{\mathcal W} = \mathcal N \; \Leftrightarrow\; -A+B+C+D=-A-B-C+D= 0 \; \Leftrightarrow\; C = -B, D=A.\]
\[{\mathcal W}^{-1}{\mathcal N}{\mathcal W} = \mathcal N \;\; \Leftrightarrow\;\; \mathcal N = \bp A & B \\ -B & A\ep.\]
\end{enumerate}

\end{example}

\begin{example}\label{anti-commuting}
Using the standard block operator trick, one can convert anti-commutator relations into commutator relations as follows.
Let $X, Y$ be self-adjoint. Then $\mathcal X = \bp X & 0 \\ 0 & -X\ep$ and $\mathcal Y = \bp 0 & Y \\ Y & 0\ep$ are self-adjoint. Moreover,
\[\mathcal X \mathcal Y = \bp 0& XY \\ -XY & 0\ep, \;\; \mathcal Y \mathcal X = \bp 0&-YX  \\ YX & 0\ep\]
so $\|XY+YX\| = \|\mathcal X \mathcal Y-\mathcal Y \mathcal X\|$ and $\|XY-YX\| = \|\mathcal X \mathcal Y+\mathcal Y \mathcal X\|$. This tells us that if we know certain results for commutators of structured operators then we also know them for anti-commutators and if we know certain results for anti-commutators of structured operators then we also know them for commutators as well.

Let $\mathcal N \in M_2(\mathcal A)$ viewed as a $2\times 2$ block matrix with blocks $N_{ij} \in \mathcal A$.
Let \begin{equation*}
\mathcal W_X = \bp I & 0 \\ 0 & - I\ep, \;\;\label{W_Y}\mathcal W_Y = \bp 0 & I \\ I & 0\ep.
\end{equation*}
These unitaries anti-commute so the symmetry maps $\varphi_X(\mathcal N) = \mathcal W_X^{-1}\mathcal N\mathcal W_X$, $\varphi_Y(\mathcal N) = \mathcal W_Y^{-1}\mathcal N\mathcal W_Y$ commute. These unitaries are their own inverses, so $\varphi_X$ and $\varphi_Y$ have order two.

Moreover by Example \ref{2x2structure}, for any self-adjoint $\mathcal X, \mathcal Y \in M_{2}(\mathcal A)$: 
\begin{align}\nonumber
\mathcal X = \bp X & 0 \\ 0 & -X\ep \;\;&\Leftrightarrow\;\; \varphi_X(\mathcal X)=\mathcal X,\; \varphi_Y(\mathcal X)=-\mathcal X,\\
\mathcal Y = \bp 0 & Y \\ Y & 0\ep \;\;&\Leftrightarrow\;\; \varphi_X(\mathcal Y)=-\mathcal Y,\; \varphi_Y(\mathcal Y)=\mathcal Y.
\label{anti-commuting characterization}
\end{align}

This provides a framework to apply the structured Lin's theorem for two self-adjoint operators with antisymmetries to $\mathcal X, \mathcal Y$. We can apply the Symmetry Bootstrap of \cite[Theorem 7.1]{herrera2024constructing} to $\mathcal X, \mathcal Y$ with $\varphi_Y$ then to $\mathcal Y, \mathcal X$ with $\varphi_X$ as we will do with block unitaries in the next section. However, we will use a simpler method using Lin's theorem with dihedral symmetries of \cite[Theorem 8.2]{herrera2024constructing}.
\end{example}

\begin{example}
Note that satisfying 
\[\varphi_X(\mathcal N) = \mathcal N^\ast, \;\; \varphi_Y(\mathcal N)=-\mathcal N^\ast.\]
is equivalent to satisfying the dihedral ($D_2$) symmetries
\[\varphi_1(\mathcal N)=-\mathcal N, \;\; \phi_1(\mathcal N)=\mathcal N\]
for $\varphi_1=\varphi_X\circ\varphi_Y$ linear multiplicative and $\phi_1=\varphi_{X\ast}$ conjugate-linear anti-multiplicative. Because $\varphi_1, \phi_1$ also commute and both have order two, they satisfy the dihedral relation $\phi_1\circ \varphi_1 \circ \phi_1 = \varphi_1^{-1}$.
\end{example}

We first need the following definitions from \cite{herrera2024projection}:
\begin{defn}
Let $\mathcal A$ be a von Neumann algebra and $S$ is a collection of weakly continuous symmetry maps on $\A$. We say that $(\mathcal A, S)$ has TR-rank 1 if any $S$-symmetric element of $\mathcal A$ can be approximated by invertible $S$-symmetric elements.

For $S$ containing linear symmetry maps, we say that $(\mathcal A, S)$ is admissible if $S$ either only contains linear multiplicative symmetry maps or $S$ also contains a single linear anti-multiplicative symmetry map that commutes with all other symmetry maps in $S$ (which are necessarily multiplicative).
\end{defn}
We extend this definition to include symmetries acting on block operators:
\begin{defn}
Suppose that $\varphi$ is a symmetry map on a $C^\ast$-algebra $\mathcal A$. Define $\varphi\otimes \operatorname{id}$ by
\[(\varphi\otimes \operatorname{id})(\bp A & B \\ C & D\ep) = \bp \varphi(A) & \varphi(B) \\ \varphi(C) & \varphi(D)\ep\] and
$\varphi\otimes T$ by
\[(\varphi\otimes T)(\bp A & B \\ C & D\ep) = \bp \varphi(A) & \varphi(C) \\ \varphi(B) & \varphi(D)\ep.\]
\end{defn}

\begin{remark}
If $\varphi$ is a multiplicative symmetry map on $\mathcal A$, then $\varphi\otimes \operatorname{id}$ is a multiplicative symmetry map on $M_2(\mathcal A)$. 
If $\varphi$ is an anti-multiplicative symmetry map on $\mathcal A$, then $\varphi\otimes T$ is an anti-multiplicative symmetry map on $M_2(\mathcal A)$.
\end{remark}

\begin{remark}
If $\varphi_1, \varphi_2, \varphi_3, \varphi_4$ commute then $\varphi_1 \otimes \operatorname{id}$, $\varphi_2 \otimes \operatorname{id}$, $\varphi_3 \otimes T$, $\varphi_3 \otimes T$ commute. 

The unitaries 
$\mathcal W_X, \mathcal W_Y$ are their own inverses and are $\varphi\otimes \operatorname{id}$-symmetric and $\varphi\otimes T$-symmetric for any $\R$-linear $\varphi$ map that is multiplicative or anti-multiplicative. So, the symmetry maps $\varphi_X$, $\varphi_Y$ commute with any $\varphi\otimes \operatorname{id}$, $\varphi\otimes T$.  
\end{remark}

\begin{defn}
We say that $(\mathcal A, S)$ has 2-TR-rank 1 if it has TR-rank 1 and if 
\[(M_2(\mathcal A), \{\varphi\otimes \operatorname{id}: \varphi\in S \mbox{ multiplicative}\}\cup \{\varphi\otimes T: \varphi\in S \mbox{ anti-multiplicative}\})\] has TR-rank 1.
\end{defn}

If we were to abuse notation and define $\varphi$ on $M_2(\mathcal A)$ to be $\varphi\otimes T$ when $\varphi$ is multiplicative and $\varphi\otimes \operatorname{id}$ when $\varphi$ is anti-multiplicative, the second condition for being  2-TR-rank 1 can be expressed by saying that $(M_2(\mathcal A), S)$ has TR-rank 1.

\begin{remark}
By \cite[Theorem 3.4]{herrera2024projection}, if $\mathcal A$ is finite dimensional then it has 2-TR-Rank 1.
\end{remark}
We now prove the stability of the relation $XY=-YX$, for $X=X^\ast, Y=Y^\ast$ with linear symmetries.
\begin{thm}\label{anti-commuting thm}
There exists a function $\epsilon(\delta)$ with $\lim_{\delta \to 0^+}\epsilon(\delta)=0$ with the following property.

Suppose that $\mathcal A$ is a von Neumann algebra and $S$ is an admissible collection of weakly continuous linear symmetry maps on $\A$ so that $(\mathcal A, S)$ has TR-rank 1.

Then for any $S$-symmetry self-adjoint contractions $X, Y$, there exist anti-commuting $S$-symmetry self-adjoint $X', Y'$ such that
\[\|X'-X\|, \|Y'-Y\| \leq \epsilon(\|XY+YX\|).\]
If every symmetry in $S$ is multiplicative then we can take $\epsilon(\delta) = Const.\delta^{1/8}$.
\end{thm}
\begin{proof}
Consider the self-adjoint operators $\mathcal X, \mathcal Y\in M_2(\mathcal A)$ and linear multiplicative symmetry maps $\varphi_X, \varphi_Y$ on $M_2(\mathcal A)$ as defined in Example \ref{anti-commuting}.

Let $\mathcal N = \mathcal X+ i\mathcal Y = \bp X & iY \\ iY & -X\ep$. Then $\mathcal N$ is almost normal: 
\[\|[\mathcal N^\ast,\mathcal N]\| = 2\|[\mathcal X, \mathcal Y]\|=2\|XY+YX\|\]
and $\mathcal N$ satisfies the symmetries:
\[(\varphi\otimes \operatorname{id})(\mathcal N)=\mathcal N, \; \varphi\in S \mbox{ multiplicative},\] 
\[(\varphi\otimes T)(\mathcal N)=\mathcal N, \; \varphi\in S \mbox{ anti-multiplicative},\]
and
\[\varphi_X(\mathcal N) =\mathcal N^\ast, \;\; \varphi_Y(\mathcal N)=-\mathcal N^\ast.\]

The symmetry maps $\varphi\otimes \operatorname{id}, \varphi\otimes T, \varphi_X, \varphi_Y$ commute. Note that satisfying 
\[\varphi_X(\mathcal N) = \mathcal N^\ast, \;\; \varphi_Y(\mathcal N)=-\mathcal N^\ast.\]
is equivalent to satisfying the dihedral ($D_2$) symmetries
\[\varphi_1(\mathcal N)=-\mathcal N, \;\; \phi_1(\mathcal N)=\mathcal N\]
for $\varphi_1=\varphi_X\circ\varphi_Y$ linear multiplicative and $\phi_1=\varphi_{X\ast}$ conjugate-linear anti-multiplicative. Because $\varphi_1, \phi_1$ also commute and both have order two, they satisfy the dihedral relation $\phi_1\circ \varphi_1 \circ \phi_1 = \varphi_1^{-1}$.

So, applying \cite[Theorem 8.2]{herrera2024projection} to $\mathcal N$, we obtain a normal $\mathcal N' \in M_2(\mathcal A)$ that satisfies 
 \[(\varphi\otimes \operatorname{id})(\mathcal N')=\mathcal  N', \; \varphi\in S \mbox{ multiplicative},\] 
\[(\varphi\otimes T)(\mathcal N')=\mathcal N', \; \varphi\in S \mbox{ anti-multiplicative},\]
and
\[\varphi_X(\mathcal N') = (\mathcal N')^\ast, \;\; \varphi_Y(\mathcal N')=-(\mathcal N')^\ast.\]
Writing 
\[\mathcal X' = \Re(\mathcal N'), \;\; \mathcal Y' = \Im(\mathcal N'),\]
we see that there exist commuting $S$-symmetric self-adjoint $X', Y'\in \A$ such that 
\[\mathcal X' = \bp X' & 0 \\ 0 & -X'\ep, \;\; \mathcal Y' = \bp 0 & Y'\\ Y' & 0\ep,\]
In particular, $X', Y'\in \mathcal A$ are the desired anti-commuting operators.
\end{proof}

If we apply this result with the transpose, then we obtain the other result stated in the introduction.

\section{Unitaries that Almost Commute up to a Rational Phase}

We apply this same method for unitaries that almost commute up to a rational phase, $\omega$. Before that, we make some definitions and observations.

\begin{example}\label{shiftStrcture}
Let $\mathcal N \in M_m(\mathcal A)$ viewed as a $m\times m$ block matrix with blocks $N_{ij} \in \mathcal A$. Define $\mathcal W$ similarly with blocks $W_{ij}$ with $\mathcal W$ additionally being unitary.

\begin{enumerate}[label=(\roman*)]
\item Let $\omega\in\C$ have order $m$  and 
\[\mathcal W_\omega = \diag(1, \omega, \dots, 
\omega^{m-1}) \otimes I_{m} = \bp 
I &         & & \\
  & \omega I &  &\\
  &        & \ddots & \\
  &           &        & \omega^{m-1}I
\ep.\] 
Then 
$\mathcal W_\omega^{-1}{\mathcal N}\mathcal W_\omega$ has its $ij$th block equal to $\omega^{j-i}N_{ij}$. In particular, $\mathcal W_\omega^{-1}{\mathcal N}{\mathcal W_\omega}={\mathcal N}$ if and only ${\mathcal N}$ is block diagonal:
\[\mathcal N  = \bp 
N_{1,1} &         & & \\
  & N_{2,2} &  &\\
  &        & \ddots & \\
  &           &        & N_{m,m}
\ep\]
and $\mathcal W_\omega^{-1}{\mathcal N}{\mathcal W_\omega} = \omega{\mathcal N}$ if and only if all the blocks of $\mathcal N$ are zero except those directly above the diagonal (where the $m,1$ entry is viewed as being ``above'' the $1,1$ entry due to cyclic entry labeling):
\[\mathcal N = \bp 
0 & N_{1,2}  &         & &          \\
 & 0 &   N_{2,3}      &&            \\
  &  & 0   & \ddots&            \\
  &   &   & \ddots   && N_{m-1,m}   \\
N_{m,1}  &   &        &&&   0 \\
\ep.\]

\item Let $\zeta$ have order dividing $m$ and let the unitary matrix $\mathcal W_m$ be the block bilateral forward shift operator:
\[\mathcal W_m = \bp 
0 &   &         &          &I \\
I & 0 &         &          &  \\
  & I & 0       &          &  \\
  &   & \ddots  & \ddots   &  \\  
  &   &         &     I    &0 \\
\ep.\]
Then $\mathcal W_m^{-1}\mathcal N \mathcal W_m$ has $ij$th block equal to $N_{i+1,j+1}$ so that conjugating by $\mathcal W_m$ shifts the blocks in each bilateral diagonal upward along that diagonal cyclically.

Consequently, $\mathcal W_m^{-1}\mathcal N \mathcal W_m=\mathcal N$ if and only if all the blocks on each bilateral diagonal of $\mathcal N$ are identical. Likewise, $\mathcal W_m^{-1}\mathcal N \mathcal W_m=\zeta\mathcal N$ if and only if the entries satisfy $N_{i+1,j+1}=\zeta N_{i,j}$.
\end{enumerate}
\end{example}

\begin{example}\label{commute up to phase block translation}
Let $U, V$ be unitary matrices in $M_n(\C)$. Define
\begin{equation}\label{unitaryBlock}
\mathcal U = \bp 
U &         & & & \\
  & \omega U &  & & \\
  &        & \ddots & & \\
    &           &        &  \omega^{m-2}U &\\
  &           &        &  & \omega^{m-1}U
\ep, \;\;\; 
\mathcal V = \bp 
0 & V  &         & &          \\
 & 0 &   V      &&            \\
  &  & 0   & \ddots&            \\
  &   &   & \ddots   && V   \\
V  &   &        &&&   0 \\
\ep.
\end{equation}
Then
\[\mathcal U \mathcal V = 
\bp 
0 & UV  &         & &          \\
 & 0 &  \omega UV      &&            \\
  &  & 0   & \ddots&            \\
  &   &   & \ddots   && \omega^{m-2}UV   \\
\omega^{m-1}UV  &   &        &&&   0 \\
\ep.\]
\[\mathcal V\mathcal U = \bp 
0 &\omega VU  &         & &          \\
 & 0 &   \omega^2 VU      &&            \\
  &  & 0   & \ddots&            \\
  &   &   & \ddots   && \omega^{m-1}VU   \\
VU  &   &        &&&   0 \\
\ep.\]
So, $\|\mathcal U\mathcal V -\mathcal V\mathcal U\| = \|UV - \omega VU\|$ and $\|\omega \mathcal U\mathcal V -\mathcal V\mathcal U\| = \|UV - VU\|$.

Moreover, $\mathcal U^{-1}\mathcal V^{-1}\mathcal U\mathcal V = (\mathcal V\mathcal U)^{-1}(\mathcal U\mathcal V)$ is block diagonal with blocks $\omega^{-1} U^{-1}V^{-1}UV$.
So, by \cite{exel1991invariants}, for $\omega^{-1} U^{-1}V^{-1}UV$ close enough to $I$ and if 
\[\Tr[\log(\mathcal U^{-1}\mathcal V^{-1}\mathcal U\mathcal V)]=m\Tr[\log(\omega^{-1} U^{-1}V^{-1}UV)]=0\] then $\mathcal U, \mathcal V$ are nearby \emph{some} commuting unitaries. 

Note that the formula,
\[\Tr[\log(\omega^{-1} U^{-1}V^{-1}UV)]=0\]
is formally equal to the normalized trace version of the Exel trace formula:
\[\frac{1}{2\pi i}\tau(\log(UVU^{-1}V^{-1}))=\theta\]
for $\theta \in (-1/2, 1/2)$ where $\omega = e^{2\pi i \theta}$ as in \cites{hua2015rotation}. Like that of \cite{pedersen1998stability}, it does not require the use of different branches of the logarithm for certain values of $\theta$ due to the phase $\omega$ being inside the logarithm as in \cite[Theorem 4.14]{pedersen1998stability} or the case of $\theta = 1/2$ in \cite[Theorem 5.4]{hua2015rotation}. 
\end{example}

\begin{example}
Let $\mathcal W_\omega$ and $\mathcal W_m$ be as in Example \ref{shiftStrcture}. 
Let $\varphi_\omega(\mathcal N) = \mathcal W_\omega^{-1}\mathcal N\mathcal W_\omega$ 
and $\varphi_m(\mathcal N) = \mathcal W_m^{-1}\mathcal N\mathcal W_m$.

A unitary $\mathcal U$ has the block structure in the first part of Equation (\ref{unitaryBlock}) if and only iff  $\varphi_\omega(\mathcal U)={\mathcal U}$ and $\varphi_m(\mathcal U)=\omega\mathcal U$.
A unitary $\mathcal V$ has the block structure in the second part of Equation (\ref{unitaryBlock}) if and only iff  $\varphi_\omega(\mathcal V)=\omega{\mathcal V}$ and $\varphi_m(\mathcal V)=\mathcal V$.

Likewise, $\mathcal W_m^{-1} \mathcal W_\omega \mathcal W_m = \omega \mathcal W_\omega$. So, $\mathcal W_\omega \mathcal W_m = \omega \mathcal W_m \mathcal W_\omega$. Because  $\mathcal W_\omega, \mathcal W_m$ commute up to a phase, their associated symmetry maps $\varphi_\omega, \varphi_m$ commute because they are defined by conjugation.

The symmetry maps $\varphi_\omega, \varphi_m$ also commute with any $\varphi \otimes \operatorname{id}_m$ linear multiplicative map that is extended to $\mathcal M_m(\mathcal A)$ by applying $\varphi$ to each entry of matrices in $M_m(\mathcal A)$.

\end{example}

We apply \cite[Theorem 7.6]{herrera2024projection} twice to obtain the following result which implies the dilation result stated in the introduction.
\begin{thm}\label{Up to Phase Bootstrap}
Let $\mathcal A$ be a von Neumann algebra and $S$ a collection of weakly continuous linear multiplicative symmetry maps on $\mathcal A$. Let $U, V \in \mathcal A$ be $S$-symmetric unitaries and $\omega\in \C$ have order $m$. 

Suppose that for the $S\otimes \operatorname{id}$-symmetric unitaries $\mathcal U, \mathcal V\in M_m(\mathcal A)$ defined by (\ref{unitaryBlock}), there exist commuting $S\otimes \operatorname{id}$-symmetric unitaries $\mathcal U', \mathcal V'\in M_m(\mathcal A)$ and
\[\varepsilon = \max(\|\mathcal U'-\mathcal U\|, \|\mathcal V'-\mathcal V\|).\]
Then there exist $S$-symmetric unitaries $U'', V''\in\mathcal A$ such that
\[U''V'' = \omega V''U'',\]
\[\|U''-U\|, \|V''-V\| \leq C_m\varepsilon^{1/4}.\]
\end{thm}
\begin{proof}
We apply the symmetry bootstrap result \cite[Theorem 7.6]{herrera2024projection} with symmetry case (ii) to $\mathcal U, \mathcal V$ with linear symmetry maps $S_0=S\otimes \operatorname{id}\cup \{\varphi_\omega\}$ and linear symmetry map $\varphi_m$ with order $m$ because
$\varphi_m(\mathcal U)=\omega\mathcal U$ and $\varphi_m(\mathcal V)=\mathcal V$ . 

We obtain $\mathcal U_0, \mathcal V_0$ commuting with $\mathcal U_0$ unitary being $S_0$-symmetric, $\varphi_m(\mathcal U_0)=\omega\mathcal U_0$, $\varphi_m(\mathcal V_0)=\mathcal V_0$, and
\[\|\mathcal U_0-\mathcal U\|, \|\mathcal V_0-\mathcal V\| \leq Const_m\cdot \varepsilon^{1/2}.\]
Moreover, the map sending $\mathcal V$ to $\mathcal V_0$ commutes with each symmetry map in $S_0$. Therefore, 
$\varphi_\omega(\mathcal V_0)=\omega{\mathcal V_0}$ because $\varphi_\omega(\mathcal V)=\omega{\mathcal V}$ and $\mathcal V_0$ is $S$-symmetric because $\mathcal V$ is $S$-symmetric.
We then perturb $V_0$ to a unitary. By \cite[Lemma 8.1]{herrera2024projection}, it maintains its symmetries and phase symmetries.

We now apply the symmetry bootstrap result \cite[Theorem 7.6]{herrera2024projection} with symmetry case (ii) to $\mathcal V_0, \mathcal U_0$ (note the order) with linear symmetry maps $S''=S\otimes \operatorname{id}\cup \{\varphi_m\}$ and linear symmetry map $\varphi_\omega$ with order $m$ because
$\varphi_\omega(\mathcal V_0)=\omega\mathcal V_0$ and $\varphi_\omega(\mathcal U_0)=\mathcal U_0$.

We obtain $\mathcal V'', \mathcal U''$ commuting with $\mathcal V''$ unitary being $S''$-symmetric, $\varphi_\omega(\mathcal V'')=\omega\mathcal V''$, $\varphi_\omega(\mathcal U'')=\mathcal U''$, and
\[\|\mathcal U_0-\mathcal U\|, \|\mathcal V_0-\mathcal V\| \leq Const_m\cdot \left(Const_m\cdot \varepsilon^{1/2}\right)^{1/2}=Const.\, \varepsilon^{1/4}.\]

Moreover, the map sending $\mathcal U_0$ to $\mathcal U''$ commutes with each symmetry map in $S''$. Therefore, 
$\varphi_m(\mathcal U'')=\omega{\mathcal U''}$ because $\varphi_m(\mathcal U'')=\omega{\mathcal U''}$ and $\mathcal U''$ is $S$-symmetric because $\mathcal U_0$ is $S$-symmetric.
We then perturb $U_0$ to a unitary. By \cite[Lemma 8.1]{herrera2024projection}, it maintains its symmetries and phase symmetries.

Because $\varphi_\omega(\mathcal U'')={\mathcal U''}$ and $\varphi_m(\mathcal U'')=\omega\mathcal U''$, the unitary $\mathcal U''\in M_m(\mathcal A)$ has the block structure in the first part of Equation (\ref{unitaryBlock}). Let $U''\in\mathcal A$ be the $(1,1)$ block of $\mathcal U''$.

Because  $\varphi_\omega(\mathcal V'')=\omega{\mathcal V''}$ and $\varphi_m(\mathcal V'')=\mathcal V''$, the unitary $\mathcal V''\in M_m(\mathcal A)$ has the block structure in the second part of Equation (\ref{unitaryBlock}). Let $V''\in \mathcal A$ be the $(1,2)$ block of $\mathcal V''$.

By Example \ref{commute up to phase block translation}, these are the desired unitaries that commute up to a phase.
\end{proof}

This result applied to Theorem \ref{UtoHtoUThm} does not give us the exponent in (\ref{phaseKS}) because we did not make use of the one automatic bootstrap symmetry inherent in Theorem \ref{UtoHtoUThm}. We do that now. 
\begin{thm}\label{almost commuting unitaries up to a phase}
Let $U, V \in M_d(\C)$ be unitaries and $\omega$ a root of unit of order $m$. Let $\mathcal U, \mathcal V$ be  defined by (\ref{unitaryBlock}).

If $B(\mathcal U,\mathcal V)$ has signature zero, then there exist unitaries $U', V'\in\mathcal M_d(\C)$ such that
\[U'V' = \omega V'U',\]
\[\|U'-U\|, \|V'-V\| \leq C_m\|UV-\omega VU\|^{1/16}.\]
\end{thm}
\begin{proof}
We will use the same set up as in the proof of Theorem \ref{Up to Phase Bootstrap} except that we apply Theorem \ref{UtoHtoUThm} to $\mathcal U, \mathcal V$ because
$\varphi_m(\mathcal U)=\omega\mathcal U$ and $\varphi_m(\mathcal V)=\mathcal V$ and we use $S = \emptyset$ so $S_0=\{\varphi_\omega\}$.

We obtain $\mathcal U_0, \mathcal V_0$ commuting unitaries with $\mathcal U_0$ being $S_0$-symmetric, $\varphi_m(\mathcal U_0)=\omega\mathcal U_0$, $\varphi_m(\mathcal V_0)=\mathcal V_0$, and
\[\|\mathcal U_0-\mathcal U\|, \|\mathcal V_0-\mathcal V\|  \leq C_\omega\|[\mathcal U,\mathcal V]\|^{1/8}=C_\omega\|UV-\omega VU\|^{1/8},\]
Moreover, the map sending $\mathcal V$ to $\mathcal V_0$ commutes with each symmetry map in $S_0$. Therefore, 
$\varphi_\omega(\mathcal V_0)=\omega{\mathcal V_0}$ because $\varphi_\omega(\mathcal V)=\omega{\mathcal V}$.
We then perturb $V_0$ to a unitary. By \cite[Lemma 8.1]{herrera2024projection}, it maintains its symmetries and phase symmetries.

The proof then continues as the proof of Theorem \ref{Up to Phase Bootstrap}.
\end{proof}

\newpage

\vspace{1in}

\textbf{ACKNOWLEDGEMENTS}. The author would like to thank Eric A. Carlen for support during the writing of this paper prior to 2024. We would also like to thank Malte Gerhold for some feedback on known results and ideas behind Remark \ref{Infinite Dilation}.
This research was partially supported by NSF grant DMS-1764254.

\renewcommand{\biblistfont}{%
	\normalfont
	\normalsize
}

\phantomsection

\addcontentsline{toc}{chapter}{Bibliography}

\bibliography{almostcommuting.bib}

@article{choi1988almost,
  title={Almost commuting matrices need not be nearly commuting},
  author={Choi, Man Duen},
  journal={Proceedings of the American Mathematical Society},
  volume={102},
  number={3},
  pages={529--533},
  year={1988}
}

@article{davidson1985almost,
  title={Almost commuting {H}ermitian matrices},
  author={Davidson, Kenneth R},
  journal={Mathematica Scandinavica},
  volume={56},
  number={2},
  pages={222--240},
  year={1985},
  publisher={JSTOR}
}

@article{dor2025almost,
  title={On almost commuting unitary matrices},
  author={Dor-On, Adam and Hall, Lucas and Kachkovskiy, Ilya},
  journal={arXiv preprint arXiv:2510.03674},
  year={2025}
}

@article{eilers1999morphisms,
  title={Morphisms of extensions of {$C^\ast$}-algebras: pushing forward the {B}usby invariant},
  author={Eilers, S{\o}ren and Loring, Terry A and Pedersen, Gert K},
  journal={Advances in Mathematics},
  volume={147},
  number={1},
  pages={74--109},
  year={1999},
  publisher={Elsevier}
}

@article{enders2019almost,
  title={Almost commuting matrices, cohomology, and dimension},
  author={Enders, Dominic and Shulman, Tatiana},
  journal={arXiv preprint arXiv:1902.10451},
  year={2019}
}

@article{exel1991invariants,
  title={Invariants of almost commuting unitaries},
  author={Exel, Ruy and Loring, Terry A},
  journal={Journal of Functional Analysis},
  volume={95},
  pages={764--76},
  year={1991}
}

@article{exel1993soft,
  title={The soft torus and applications to almost commuting matrices},
  author={Exel, Ruy},
  journal={Pacific Journal of Mathematics},
  volume={160},
  number={2},
  pages={207--217},
  year={1993},
  publisher={Mathematical Sciences Publishers}
}

@article{exel1989almost,
  title={Almost commuting unitary matrices},
  author={Exel, Ruy and Loring, Terry A},
  journal={Proceedings of the American Mathematical Society},
  volume={106},
  number={4},
  pages={913--915},
  year={1989}
}

@article{friis1996almost,
  title={Almost commuting self-adjoint matrices-a short proof of {H}uaxin {L}in's theorem.},
  author={Friis, Peter and R{\o}rdam, Mikael},
  year={1996},
  publisher={Journal f{\"u}r die reine und angewandte Mathematik 479}
}

@article{gerhold2024dilation,
  title={Dilation distance and the stability of ergodic commutation relations},
  author={Gerhold, Malte and Shalit, Orr},
  journal={arXiv preprint arXiv:2406.05864},
  year={2024}
}

@article{gong1998almost,
  title={Almost multiplicative morphisms and almost commuting matrices},
  author={Gong, Guihua and Lin, Huaxin},
  journal={Journal of Operator Theory},
  pages={217--275},
  year={1998},
  publisher={JSTOR}
}

@book{loukas2014classical,
  title={Classical Fourier Analysis},
  author={Loukas Grafakos, Loukas Grafakos},
  year={2014},
  publisher={Springer}
}

@article{hastings2009making,
  title={Making almost commuting matrices commute},
  author={Hastings, Matthew B},
  journal={Communications in Mathematical Physics},
  volume={291},
  number={2},
  pages={321--345},
  year={2009},
  publisher={Springer}
}

@article{hastings2011making,
  title={Making Almost Commuting Matrices Commute},
  author={Hastings, Matthew B},
  journal={arXiv preprint arXiv:0808.2474},
  year={2011}
}

@article{hastings2010almost,
  title={Almost commuting matrices, localized {W}annier functions, and the quantum {H}all effect},
  author={Hastings, Matthew B and Loring, Terry A},
  journal={Journal of mathematical physics},
  volume={51},
  number={1},
  year={2010},
  publisher={AIP Publishing}
}

@article{hastings2011topological,
  title={Topological insulators and {$C^\ast$}-algebras: Theory and numerical practice},
  author={Hastings, Matthew B and Loring, Terry A},
  journal={Annals of Physics},
  volume={326},
  number={7},
  pages={1699--1759},
  year={2011},
  publisher={Elsevier}
}

@article{herrera2022constructing,
  title={Constructing Nearby Commuting Matrices for Reducible Representations of {$su(2)$} with an Application to {O}gata's Theorem},
  author={Herrera, David},
  journal={arXiv preprint arXiv:2212.06012},
  year={2022}
}

@article{herrera2020hastings,
  title={On {H}astings' approach to {L}in's Theorem for Almost Commuting Matrices},
  author={Herrera, David},
  journal={arXiv preprint arXiv:2011.11800},
  year={2020}
}

@article{herrera2024projection,
  title={A Projection Characterization and Symmetry Bootstrap for Elements of a von {N}eumann Algebra that are Nearby Commuting Elements},
  author={Herrera, David},
  journal={arXiv preprint arXiv:},
  year={2024}
}

@phdthesis{herrera2024constructing,
  title={Constructing nearby commuting matrices for ogata's theorem on macroscopic observables},
  author={Herrera, David},
  year={2024},
  school={Rutgers University-School of Graduate Studies},
  url={https://doi.org/doi:10.7282/t3-k837-yj42}
}

@article{kachkovskiy2016distance,
  title={Distance to normal elements in {$C^\ast$}-algebras of real rank zero},
  author={Kachkovskiy, Ilya and Safarov, Yuri},
  journal={Journal of the American Mathematical Society},
  volume={29},
  number={1},
  pages={61--80},
  year={2016}
}

@article{li2022vector,
  title={Vector-wise Joint Diagonalization of Almost Commuting Matrices},
  author={Li, Bowen and Lu, Jianfeng and Yu, Ziang},
  journal={arXiv preprint arXiv:2205.15519},
  year={2022}
}

@article{lin1996almost,
  title={Almost commuting selfadjoint matrices and applications},
  author={Lin, Huaxin},
  journal={Operator algebras and their applications},
  pages={193--233},
  year={1996},
  publisher={American Mathematical Society}
}

@article{lin2024almost,
  title={Almost commuting self-adjoint operators and measurements},
  author={Lin, Huaxin},
  journal={arXiv preprint arXiv:2401.04018},
  year={2024}
}

@article{lin1995almost,
  title={Almost commuting unitary elements in purely infinite simple C*-algebras},
  author={Lin, Huaxin},
  journal={Mathematische Annalen},
  volume={303},
  number={1},
  pages={599--616},
  year={1995},
  publisher={Springer}
}

@article{pedersen1998stability,
  title={Stability of anticommutation relations: an application of noncommutative CW-complexes},
  author={Pedersen, Gert K and Loring, Terry A and Eilers, S{\o}ren},
  year={1998},
  publisher={Walter de Gruyter GmbH \& Co. KG Berlin, Germany}
}

@article{loring1988k,
  title={K-theory and asymptotically commuting matrices},
  author={Loring, Terry A},
  journal={Canadian Journal of Mathematics},
  volume={40},
  number={1},
  pages={197--216},
  year={1988},
  publisher={Cambridge University Press}
}

@article{loring2015k,
  title={K-theory and pseudospectra for topological insulators},
  author={Loring, Terry A},
  journal={Annals of Physics},
  volume={356},
  pages={383--416},
  year={2015},
  publisher={Elsevier}
}

@book{loring1997lifting,
  title={Lifting solutions to perturbing problems in C*-algebras},
  author={Loring, Terry A},
  volume={8},
  year={1997},
  publisher={American Mathematical Soc.}
}

@article{loring2014quantitative,
  title={Quantitative {K}-theory related to spin {C}hern numbers},
  author={Loring, Terry A},
  journal={SIGMA. Symmetry, Integrability and Geometry: Methods and Applications},
  volume={10},
  pages={077},
  year={2014},
  publisher={SIGMA. Symmetry, Integrability and Geometry: Methods and Applications}
}

@article{loring1998matrices,
  title={When matrices commute},
  author={Loring, Terry A},
  journal={Mathematica Scandinavica},
  pages={305--319},
  year={1998},
  publisher={JSTOR}
}

@article{loring2016almost,
  title={Almost commuting self-adjoint matrices: the real and self-dual cases},
  author={Loring, Terry A and S{\o}rensen, Adam PW},
  journal={Reviews in Mathematical Physics},
  volume={28},
  number={07},
  pages={1650017},
  year={2016},
  publisher={World Scientific}
}

@article{loring2014almost,
  title={Almost commuting orthogonal matrices},
  author={Loring, Terry A and S{\o}rensen, Adam PW},
  journal={Journal of Mathematical Analysis and Applications},
  volume={420},
  number={2},
  pages={1051--1068},
  year={2014},
  publisher={Elsevier}
}

@article{loring2013almost,
  title={Almost commuting unitary matrices related to time reversal},
  author={Loring, Terry A and S{\o}rensen, Adam PW},
  journal={Communications in Mathematical Physics},
  volume={323},
  pages={859--887},
  year={2013},
  publisher={Springer}
}

@article{ogata2013approximating,
  title={Approximating macroscopic observables in quantum spin systems with commuting matrices},
  author={Ogata, Yoshiko},
  journal={Journal of Functional Analysis},
  volume={264},
  number={9},
  pages={2005--2033},
  year={2013},
  publisher={Elsevier}
}

@article{voiculescu1981remarks,
  title={Remarks on the singular extension in the {$C^\ast$}-algebra of the {H}eisenberg group},
  author={Voiculescu, Dan},
  journal={Journal of Operator Theory},
  pages={147--170},
  year={1981},
  publisher={JSTOR}
}

@article{voiculescu1983asymptotically,
  title={Asymptotically commuting finite rank unitary operators without commuting approximants},
  author={Voiculescu, Dan},
  journal={Acta Sci. Math.(Szeged)},
  volume={45},
  number={1-4},
  pages={429--431},
  year={1983}
}

@article{aleksandrov2016operator,
  title={Operator lipschitz functions},
  author={Aleksandrov, A. B. and Peller, V. V.},
  journal={Russian Mathematical Surveys},
  volume={71},
  number={4},
  pages={605},
  year={2016},
  publisher={IOP Publishing}
}

@article{hua2015rotation,
  title={Rotation algebras and the Exel trace formula},
  author={Hua, Jiajie and Lin, Huaxin},
  journal={Canadian Journal of Mathematics},
  volume={67},
  number={2},
  pages={404--423},
  year={2015},
  publisher={Cambridge University Press}
}

@article{hua2021stability,
  title={Stability of rotation relations in-algebras},
  author={Hua, Jiajie and Wang, Qingyun},
  journal={Canadian Journal of Mathematics},
  volume={73},
  number={4},
  pages={1171--1203},
  year={2021},
  publisher={Canadian Mathematical Society}
}

\Addresses

\end{document}